\documentclass[12pt]{article}
\usepackage[english]{babel}
\usepackage[utf8x]{inputenc}
\usepackage[T1]{fontenc}
\usepackage{color}
\usepackage{hyperref}
\usepackage{tcolorbox}
\usepackage{amssymb,amsmath,amsthm,amsfonts,latexsym}
\usepackage{bm,fullpage}
\newtheorem{theorem}{Theorem}

\allowdisplaybreaks[3]
\DeclareMathOperator{\bfx}{\mathbf{x}}

\DeclareMathOperator{\bff}{\mathbf{f}}

\DeclareMathOperator{\prox}{\mathbf{prox}}
\DeclareMathOperator{\bfg}{\mathbf{g}}
\DeclareMathOperator{\bfs}{\mathbf{s}}

\DeclareMathOperator{\bfa}{\mathbf{a}}

\DeclareMathOperator{\bfd}{\mathbf{d}}
\DeclareMathOperator{\bfr}{\mathbf{r}}
\DeclareMathOperator{\bfz}{\mathbf{z}}
\DeclareMathOperator{\bfy}{\mathbf{y}}
\DeclareMathOperator{\R}{\mathbb{R}}
\DeclareMathOperator{\A}{\mathcal{A}}
\DeclareMathOperator{\N}{\mathbb{N}}
\DeclareMathOperator{\F}{\mathcal{F}}

\def\bpm{\begin{pmatrix}}
\def\epm{\end{pmatrix}}

\DeclareMathOperator{\la}{\langle}
\DeclareMathOperator{\ra}{\rangle}

\DeclareMathOperator{\new}{new}
\DeclareMathOperator{\nemi}{nemi}
\DeclareMathOperator{\LY}{LY}
\DeclareMathOperator{\poly}{poly}
\newcommand{\bfsbar}{\bar{\mathbf{s}}}
\newcommand{\FmuL}{\F_{\mu,L}}
\newcommand{\FmuLRn}{\FmuL (\R^n)}

\usepackage{mathtools}
\DeclarePairedDelimiterX{\inn}[2]{\langle}{\rangle}{#1, #2}
\usepackage{physics}	% trace symbol, norm, (partial) derivatives etc.
\usepackage{graphicx}	% figures
\usepackage{xr,url,cite}
\usepackage[algoruled]{algorithm2e} % algorithm environment
\usepackage{booktabs} % nice tables
\usepackage{float}		% option to fix floats exactly where you want them
\usepackage[symbol]{footmisc} % for symbol footnotemarks

\usepackage[normalem]{ulem}
\usepackage{cancel}
\makeatletter
\newcommand{\specificthanks}[1]{\@fnsymbol{#1}}% Inserts a specific \thanks symbol
\makeatother

\def\be{\begin{equation}}
\def\ee{\end{equation}}
\def\bi{\begin{itemize}}
\def\ei{\end{itemize}}

\begin{document}
\title{Analysis of Optimization Algorithms via Sum-of-Squares}

\author{
	Sandra~S.~Y.~Tan  \and Antonios Varvitsiotis \and Vincent~Y.~F.~Tan\thanks{  Sandra~S.~Y.~Tan was with the   Department of Electrical and Computer Engineering, National University of Singapore    (\url{sandra_tsy@u.nus.edu}). 
 Antonios Varvitsiotis was with the Department of Electrical and Computer Engineering and  Department of Industrial Systems Engineering and Management, National University of Singapore (\url{avarvits@gmail.com}).
	Vincent Tan is with the   Department of Electrical and Computer Engineering and  Department of Mathematics, National University of Singapore   (\url{vtan@nus.edu.sg}).}
}
	
	%  \newline\hspace*{.49cm}  }

\maketitle

\begin{abstract}
We introduce a new framework for unifying and systematizing the performance analysis of first-order black-box optimization algorithms for unconstrained convex minimization. % over finite-dimensional Euclidean spaces.
The low-cost iteration complexity enjoyed by first-order  algorithms renders them particularly relevant for applications in machine learning and large-scale data analysis. % \red{however}, existing proofs of convergence 
%of such optimization algorithms 
%consist mostly of  case-by-case analyses.
%On the other hand, 
Relying  on sum-of-squares (SOS) optimization, we  introduce a hierarchy of semidefinite programs  that  give increasingly better   convergence bounds for higher levels of the hierarchy. 
Alluding to  the power of the  
SOS hierarchy,  we show that the (dual of the) first level corresponds to  the Performance Estimation Problem   (PEP)
% an area of study 
 introduced  by Drori and Teboulle  [{\em Math.\ Program.}, 145(1):451--482, 2014],  a powerful framework for   determining convergence  rates of   first-order  optimization algorithms.  Consequently, 
 many results obtained within the PEP framework can be reinterpreted as degree-1 SOS proofs, and thus, the SOS  framework
  % and developed further by  Taylor, Hendrickx, and Glineur [{\em Math.\ Program.}, 161(1):307--345, 2017]. Additionally,  
%as the eve  that are increasingly  better for higher levels,.
 %(SDPs), for which higher levels of the hierarchy lead to improved bounds. 
provides  a  promising  new approach for   certifying  
improved  rates of convergence  by means of higher-order SOS certificates. To determine analytical rate bounds, in this work we  use the first level of the SOS hierarchy  and derive  new  result{s} for noisy gradient descent with   inexact line search methods (Armijo, Wolfe,  and Goldstein). % and 
%recover  several  convergence bounds for four widely-used first-order optimization algorithms. 
%in a unified manner, 
%Consequently, our results (only in the  )
%\cite{Drori2014,Taylor}.}
%Illustrating the usefulness
\end{abstract}
%\\
% \par
% \noindent  Communicated by Marc Teboulle

%\tableofcontents
\section{Introduction}
The pervasiveness of machine learning and big-data analytics throughout most academic fields and industrial domains has triggered renewed interest in convex optimization, the subfield of mathematical optimization that is concerned with minimizing a convex objective function over a convex set of decision variables. Of particular relevance for solving large-scale convex optimization problems with low accuracy requirements are first-order algorithms, defined as iterative algorithms that only use (sub)gradient information. 

There exists extensive literature on the convergence analysis of first-order optimization algorithms with respect to various performance metrics; see, e.g., \cite{nonlinear_dimitri, boyd, CG_nstepquad, CGoverview} and the references therein. 
However, existing convergence results typically rely on case-by-case analyses and cannot be understood  by a common guiding principle. In this work we introduce a unified framework for deriving worst-case upper bounds
 %over certain function classes of functions 
 on the convergence rates of first-order  optimization algorithms, through the use of sum-of-squares (SOS) optimization.

SOS optimization is an active research area with important practical applications; see,~e.g., \cite{BPT, SOS-Parrilo, SOS-Lasserre, SOS-Laurent, SOStextbook}. The key idea underlying SOS optimization is to use  semidefinite programming (SDP) relaxations for certifying the nonnegativity of a polynomial over a set defined by polynomial (in)equalities.  This allows to construct hierarchies of SDPs that approximate the optimal value of arbitrary polynomial optimization problems.
%the true solution.

To illustrate the main ingredients of our approach, consider the problem of  
minimizing a convex function $f:\R^n\to \R$ over $\R^n$, %its entire domain, 
i.e., 
$\min_{\bfx \in \R^n} f(\bfx),
$ and let $\bfx_*$ be a global minimizer. Any solution strategy  
 entails choosing a black-box algorithm $\A$ that generates a sequence of iterates~$\{\bfx_k\}_{k\ge 1}$.
Our goal is then to estimate the worst-case convergence rate of $\A$ with respect to a fixed family of functions $\F$ and an appropriate measure of performance (e.g., distance to optimality $\|\bfx_k-\bfx_*\|$ or objective function accuracy $f(\bfx_k)-f(\bfx_*))$.
For concreteness, using as performance metric the objective function accuracy and a first-order algorithm $\A$ that does not increase the objective function value at each step,   we seek to solve the following optimization~problem: 
\begin{equation}\label{eq:pepintro}
\begin{aligned} 
t_*=\text{minimize} %\min 
\ \ &  t\\
\text{subject to} \ \ & f_{k+1} - f_*\le t(f_k - f_*), \\
& \bfx_{k+1}=\A\left( \bfx_0,\dots,\bfx_k; f_0,\dots,f_k; \bfg_0,\dots, \bfg_k\right), \\
& \text{ for all } f\in \mathcal{F},
\end{aligned}
\end{equation}
where we set $f_k = f(\bfx_k)$ and $\bfg_k = \nabla f(\bfx_k)$ for all $k\ge 1$. 
As the optimization problem \eqref{eq:pepintro} is hard in general, we relax it into a tractable convex program (in fact, an SDP), in two steps. In the first step, we derive necessary conditions that are expressed as polynomial inequalities ${h_1(\bfz)\ge 0}, \ldots, {h_m(\bfz)\ge 0} $ and equalities $v_1(\bfz)=0, \ldots, v_{m'}(\bfz)=0$, in terms of the variables in
$\bfz=(f_*, f_k, f_{k+1},\bfx_*,\bfx_k, \bfx_{k+1}, \bfg_*, \bfg_k, \bfg_{k+1})$,
which are dictated by the choice of the algorithm and the corresponding class of functions. Having identified these necessary polynomial constraints, the first relaxation of the optimization problem~\eqref{eq:pepintro} is to find the minimum $t\in (0,1)$ such that the polynomial $t(f_k - f_*)-(f_{k+1} - f_*)$ is nonnegative over the semi-algebraic set 
$${K=\{\bfz : h_i(\bfz)\ge 0, \   i\in [m], \ v_j(\bfz)=0, \  j\in [m'] \}},$$ where here and throughout we use the notation $[m]=\{1,\dots, m\}$.
Nevertheless, as this second problem is also hard in general, in the second step we further relax this constraint by demanding that the nonnegativity of the polynomial $t(f_k - f_*)-(f_{k+1} - f_*) $ over $K$ is certified by an SOS decomposition: 
\begin{equation}\label{eq:putinarintro}
t(f_k - f_*)-(f_{k+1} - f_*)  = \sigma_0(\bfz) + \sum_{i=1}^m \sigma_i(\bfz)h_i(\bfz) + \sum_{j=1}^{m'} \theta_j(\bfz)v_j(\bfz), 
\end{equation}
where the {$\sigma_i(\bfz)$'s} are SOS polynomials and the {$\theta_j(\bfz)$'s} are arbitrary polynomials. Clearly, expression~\eqref{eq:putinarintro} certifies that $t(f_k - f_*)-(f_{k+1} - f_*) $ is nonnegative over the semi-algebraic set $K$. Furthermore, once the degree of the $\sigma_i$'s and the $\theta_j$'s has been fixed, the problem of finding the least $t\in (0,1)$ such that~\eqref{eq:putinarintro} holds is an instance of an SDP, and thus, it can be solved efficiently.
%Lastly, the use of SOS certificates allows to estimate the optimal rate using SDPs.

\subsection{Related Work}\label{sec:related}
%There are several works aiming to unify the convergence analyses of optimization algorithms. 

\paragraph{Performance Estimation Problem. }%\label{SEC:PEP}
Our work was  motivated by the recent framework 
%A closely related approach to deriving worst-case bounds on the performance of first-order optimization algorithms was 
introduced by Drori and Teboulle \cite{Drori2014} that  
%Specifically, the authors developed a framework that
casts the search for worst-case rate bounds as an  infinite-dimensional optimization problem:
\begin{equation} \tag{PEP}
\begin{aligned} 
\underset{f, \bfx_0,\ldots,\bfx_N,\bfx_*}{\text{maximize}} \ \ &  f(\bfx_N)-f(\bfx_*)\\
\text{subject to} \ \ &f\in \F, \\
& \bfx_{k+1}=\A\left( \bfx_0,\dots,\bfx_k; f_0,\dots,f_k;  \nabla f(\bfx_0),\dots,  \nabla f(\bfx_k)\right), \  0\le k\le N-1, \\
&  \bfx_* \text{ is a minimizer of $f$ on }\R^n,%\in X_*(f),
 \  \|\bfx_0-\bfx_*\|\le R,\\
& \bfx_0,\ldots,\bfx_N, \bfx_*\in \R^n,
\end{aligned}
\end{equation}
called the \emph{Performance Estimation Problem} (PEP).  A series of recent works has  highlighted the   PEP  as an extremely useful tool in  various settings,  including the study of worst-case guarantees for first-order optimization algorithms~\cite{Drori2014,ProxGrad,Taylor,deKlerk,TB19,THG17,dKGT17}, the design of optimal methods~\cite{Drori2014,KF18b,KF15, DT16, DT19}, and the study of  worst-case guarantees for solving  monotone inclusion problems~\cite{RTBG, Kim19,GY19,L19}.  
The PEP captures the worst-case objective function accuracy over all functions within $\F$, after $N$ iterations of the algorithm $\A$ from any starting point $\bfx_0$, which is within distance $R$ from some minimizer~$\bfx_*$.

Although  the PEP is infinite-dimensional (as its search space includes all functions in the class $\F$),  it can be transformed into an equivalent finite-dimensional problem using the {\em (smooth) convex interpolation} approach introduced in \cite{Taylor}. 
Following \cite{Taylor}, the functional constraint  $f\in\F$ is discretized by introducing $2(N+2)$ additional variables capturing the value and the gradient of the function at the points $\bfx_0,\ldots,\bfx_N,\bfx_*$. Specifically, setting $I=\{0,1,\ldots,N,*\}$,  the finite-dimensional problem 
\begin{equation} \tag{f-PEP}\label{f-PEP}
\begin{aligned} 
\underset{\{\bfx_i,\bfg_i,f_i\}_{i\in I}}{\text{maximize}} \ \ &  f_N-f_*\\
\text{subject to} \ \ & \exists f\in \F \text{ such that } f_i=f(\bfx_i), \ \bfg_i=\nabla f(\bfx_i) \ \text{ for all }  i\in I,\\
& \bfx_{k+1}=\A\left( \bfx_0,\dots,\bfx_k; f_0,\dots,f_k; \bfg_0,\dots, \bfg_k\right), \ k=0,\ldots,N-1,  \\
&  \bfg_*=0, \  \|\bfx_0-\bfx_*\|\le R,
%& \bfx_0,\ldots,\bfx_N, \bfx_*\in \R^n,
\end{aligned}
\end{equation}
with decision  variables $ \{\bfx_i,\bfg_i,f_i\}_{i\in I}$, is equivalent to the PEP in the sense that their optimal values coincide and an optimal solution to the PEP can be transformed to an optimal solution to the f-PEP (and conversely).

The seemingly simple step of reformulating the PEP into  f-PEP by discretizing  and  introducing interpolability constraints leads naturally to a powerful approach for evaluating (or upper bounding) the value of  the f-PEP. %Specifically, set $\bff  =(f_0, \ldots,f_N,f_*)$, $X = (\bfx_0 \ \dots \ \bfx_N \ \bfx_* \ \bfg_0 \ \dots$ $\ \bfg_N \ \bfg_*) \in \R^{n \times 2(N+2)}$, and $G = X^\top X$,
%$X=(\bfx_0 \  \ldots \ \bfx_N\ \bfx_*)\in \R^{n\times (N+2)}$, $P=(\bfg_0 \ \ \ldots \ \bfg_N\ \bfg_*)\in~\R^{n\times (N+2)}$
Specifically,  if interpolability with respect to $\F$ and the iterates generated by $\mathcal{A}$
%assume that a set of triples  $\{\bfx_i,\bfg_i,f_i\}_{i\in I}$  which is $\F$-interpolable and that the  iterates generated by the algorithm under consideration 
 satisfy conditions that are linear in $\bff=(f_0, \ldots,f_N,f_*)$ and   the  entries of the Gram matrix  $G = X^\top X$, where $X = (\bfx_0 \ \dots \ \bfx_N \ \bfx_* \ \bfg_0 \ \dots$ $\ \bfg_N \ \bfg_*) \in \R^{n \times 2(N+2)}$,
%\begin{itemize}
%	\item[(1)]  that are linear in $\bff$ and the inner products of $\bfx_i$ and $\bfg_i$, i.e., 
%	\begin{equation}\label{const1}
%		\la c_i,\bff\ra +  \inn{C_i}{G} \geq a_i \text{\quad or \quad } \la d_j,\bff\ra + \inn{D_j}{G} = b_i,%\bpm X^\top X & X^\top P\\ P^\top X& P^\top P \epm \le d_i, %\  i=1,\ldots, k.
%	\end{equation}
%	\red{for some vectors $c_i, d_j$ and matrices $C_i, D_j$.}
%	\item[(2)] The  iterates generated by the algorithm  $\A$~satisfy $ \bfx_N\in{\rm span}\{\bfx_0,\dots,\bfx_{N-1}; \bfg_0,\dots, \bfg_{N-1}\},$~i.e., 
%	\begin{equation}\label{algdefn}
%		\bfx_{N}=\sum_{i=0}^{N-1}\lam_i\bfx_i+\sum_{i=0}^{N-1}\mu_i\bfg_i,
%	\end{equation}
%	where the coefficients $\lam_i, \mu_i$ can be either fixed (i.e., non-adaptive step-sizes) or varying (e.g. defined by exact line search). In the fixed step-size regime, \eqref{algdefn} is equivalent to $\|\bfx_{N}-\sum_{i=0}^{N-1}\lam_i\bfx_i-\sum_{i=0}^{N-1}\mu_i\bfg_i\|^2=0$, which can be in turn written in terms of the inner products of $\bfx_i$ and $\bfg_j$, i.e., 
%	\begin{equation}\label{const2}
%		\inn{C_{{\A}}}{G} = 0,\footnote{if we put $A$ for algorithm, we should also put $f$ for function?}  %\bullet  \bpm X^\top X & X^\top P\\ P^\top X& P^\top P \epm=0.
%	\end{equation}
%	where $C_{\A}$ encodes the constraints induced by the algorithm of choice $\A$.
%	More generally, if the step-sizes are adaptive, we would need to include additional necessary constraints of the form \eqref{const1}, that are implied by the rule for choosing the step-sizes. 
%\end{itemize}
the
%Under assumptions (1) and (2), 
 value of the f-PEP is upper bounded by the  SDP defined by all necessary functional and algorithmic constraints, 
%(\eqref{const1} and~\eqref{const2} respectively), 
as well as the appropriate reformulations in terms of $\bff$ and $G$ of the optimality condition $\bfg_* = 0$ and the initialization condition
%boundedness constraint 
$\|\bfx_0-\bfx_*\| \leq R$. 

Moreover, in the case where interpolability with respect to $\F$ and the first-order method under consideration are both  linearly Gram-representable, i.e.,   exactly characterized by a finite number of constraints   that are linear in $\bff$ and   in the  entries of $G$, the corresponding SDP relaxation of f-PEP is tight, for large enough values of $n$. 

 Interpolability conditions 
%It has been shown that $\F$-interpolability  conditions 
  have been formulated exactly for various function classes, including the class of $L$-smooth and $\mu$-strongly convex functions \cite[Theorem~5]{Taylor}, indicator and support functions \cite[Section~3.3]{THG17}, smooth and nonconvex functions \cite[Section~3.4]{THG17}. In terms of the tightness of the SDP relaxation of the f-PEP, in the case where $\F$ is one of the aforementioned  function classes, and the corresponding algorithm is a fixed-step linear first-order method as defined in \cite[Definition~2.11]{THG17}
%\footnote{Online version, this is definition 6. Are you referring to journal version? Also, this probably means we need to modify equation (5), right?}
the SDP relaxation is tight, as long as $2(N+1) \leq n$ \cite[Proposition~2.6]{THG17}.

\paragraph{Integral Quadratic Constraints. }
%Lastly, we mention in passing two 
A  competing    approach that uses SDPs to analyze iterative optimization algorithms was introduced in \cite{lessard}.
%although these are not immediately relevant to this work. 
%In the  \cite{lessard}. In 
%that work,  
In this setting,  the minimizers of the function of interest are mapped to the fixed points of a discrete-time linear dynamical system with a nonlinear feedback law, whose convergence is then analyzed using integral quadratic  constraints (IQCs). 
The IQC approach allows one to derive analytical and numerical upper bounds on the convergence rates for various algorithms by solving small SDPs.
%\footnote{It seems more appropriate to say sng general about the approach and not the specific paper}.
For instance, in \cite{lessard}, algorithms considered include the gradient method, the heavy-ball method, Nesterov's accelerated method (and related variants) applied to smooth and strongly convex functions.

The line of research initiated in \cite{lessard} has been generalized further in various directions. Some notable  examples include the convergence analysis of the ADMM method \cite{NLRPJ}, the case of non-strongly convex objective functions \cite{Faz18}, the generalization to stochastic algorithms \cite{HSL}, and the design of  first-order optimization algorithms \cite{Faz18sos}. In addition, an approach drawing upon ideas from both the PEP and IQC frameworks, and comparison between these, was proposed in \cite{TVSL}.

\subsection{Summary of Results} 
In most instances where the PEP framework was applied in the literature, close inspection of the proofs of the analytic worst-case bounds reveals that they can be reinterpreted as simple, i.e., low-degree SOS certificates; see, e.g., \cite[Appendix A]{ProxGrad}, \cite[Section 3.6]{Taylor}, and~\cite[Section 4.1]{deKlerk}. 
%appeared in the derivation of their  bounds, e.g., see  \cite{deKlerk} \cite{ProxGrad}. 
This observation is the point of departure for  our work, whose aim is to unify the aforementioned results, and additionally, to make the search for the underlying SOS certificates~explicit.

As it turns out, the connection between the PEP and the SOS framework  is an instance of SDP duality. Specifically, we have mentioned that relaxing the f-PEP into an SDP requires  the functional and algorithmic constraints to imply linear constraints of  the form
\be\label{const1}
		\la c_i,\bff\ra +  \inn{C_i}{G} \geq a_i \text{\quad or \quad } \la d_j,\bff\ra + \inn{D_j}{G} = b_i,%\bpm X^\top X & X^\top P\\ P^\top X& P^\top P \epm \le d_i, %\  i=1,\ldots, k.
		\ee
		for appropriate vectors $c_i, d_j$, matrices $C_i,D_j$ and scalars $a_i, b_j$. 
		% given in~\eqref{const1}.
Nevertheless, notice that an equivalent way expressing   the constraints in  \eqref{const1} is as {\em polynomial constraints} in the variables $f_*, f_k, f_{k+1}, \bfx_*, \bfx_k, \bfx_{k+1}, \bfg_*, \bfg_k, \bfg_{k+1}$. Specifically, setting $\bfz_0 = (f_*,f_k,f_{k+1})$ and $\bfz_\ell = (x_*(\ell),x_k(\ell), x_{k+1}(\ell), g_*(\ell), g_k(\ell), g_{k+1}(\ell))$, where we use $x_*(\ell)$ to denote the $\ell$\textsuperscript{th} coordinate of $\bfx_*$ for $\ell\in [n]$, 
 %(and similarly for the other quantities indexed by $\ell$), 
the constraints in   \eqref{const1} may be equivalently expressed as
%for reasons explained in Section \ref{sec:choices}, the polynomial constraints $h_i(\bfz) \geq 0$ and $v_j(\bfz) = 0$ we use take the following form respectively:
\begin{equation}\label{const3}
	\inn{c_i}{\bfz_0} + \sum_{\ell=1}^n \bfz_\ell^\top C_i \bfz_\ell \geq a_i \quad\text{or} \quad \inn{d_j}{\bfz_0} + \sum_\ell \bfz_\ell^\top D_j \bfz_\ell = b_i,
\end{equation}
%where $\bfz_0 = (f_*,f_k,f_{k+1})$ and $\bfz_\ell = \left(x_*(\ell),x_k(\ell), x_{k+1}(\ell), g_*(\ell), g_k(\ell), g_{k+1}(\ell)\right)$, for some vectors $c_i, d_j$ and matrices $C_i, D_j$. We use $x_*(\ell)$ to denote the $\ell$\textsuperscript{th} coordinate of $\bfx_*$ for $\ell\in [n]$. Now, \eqref{const3} is equivalent to \eqref{const1} when $a_i = b_i = 0$.
i.e., as polynomials in the variables  $\bfz_0, \bfz_1, \ldots, \bfz_n$, to which we apply the  SOS~framework.  
Formalizing this connection, in Theorem \ref{sosvspep}  we show  that the dual of the first level of the SOS hierarchy is equivalent  to the PEP  when the  functional and algorithmic constraints are linearly Gram-representable.    This      allows to reinterpret  existing  rate bounds derived within the PEP framework as~{degree-1} SOS certificates.

 Nevertheless, despite its many successful applications, the PEP framework does not offer a systematic way  by which the SDP relaxation can be strengthened when  the function class under consideration  or the employed algorithm  are  not linearly Gram-representable.  Indeed, recall  that to go from the PEP to an SDP we   take  {\em two relaxation steps}. In the first step we extract necessary  (quadratic) conditions  that are dictated by the   interpolability    with respect to $\mathcal{F}$ and the algorithm~$\mathcal{A}$. In the second step, we use the identified conditions to formulate an SDP, which  gives the desired rate bounds.  Now, it is clear that if the first relaxation step is loose, then  the value of the SDP is not necessarily equal to the value of the PEP. In such a setting, there is no systematic way to strengthen  the PEP-SDP, whereas,  the sum-of-squares approach clearly provides a solution: just consider a higher level of the hierarchy. This is exactly why we believe that the sum-of-squares approach is an interesting and  complementary approach to the Gram matrix approach of Taylor et al. \cite{Taylor}.

On the other hand, the SOS approach  provides a systematic framework for finding   better (i.e., smaller) bounds on the worst-case contraction factor of descent algorithms, by using higher levels of the SOS hierarchy.  It is worth noting though that this flexibility comes at a computational cost, in the sense that the SDPs obtained via the SOS hierarchy are dimension-dependent, i.e., 
any performance certificate generated by the model only applies to functions over
a domain with a fixed dimension $n$. 

To overcome this issue,  we show in Theorem~\ref{lifting}  that in the specific setting studied in this work (cf. Section \ref{sec:choices}),  a degree-1 certificate for the univariate case  (i.e. $n=1)$ can be {\em lifted} to a degree-1 certificate for the general case $(n>1$). Nevertheless, we have  been unable  to extend this lifting procedure  for higher-order SOS certificates. As our goal is to identify analytic rates, the inability to work with general $n$ has forced us to only consider degree-1 certificates. We leave the consideration of higher degree certificates to future~work.

In terms of   using the SOS  hierarchy to derive new convergence results,  we focus   on  gradient descent  applied  to $L$-smooth, $\mu$-strongly convex functions, where the step size is chosen using  inexact line search methods. %As our aim is to derive analytic bounds, for the reasons explained above, we only consider degree-1 certificates.  
Specifically,   in Theorem~\ref{thm:armijo_result}, Theorem~\ref{thm:goldstein_result} and Theorem~\ref{thm:wolfe}  we respectively study  the  Armijo, Wolfe, and Goldstein conditions with step size selection in both the noisy and noiseless settings. Denoting by   $\delta \in [0,1)$  the noise level in the gradient  estimation  (see \eqref{eq:noise}),
our main results are the following  rate bounds:

\medskip
\noindent {\em Gradient descent with Armijo-terminated line search:}
\begin{equation*}
	f_{k+1} - f_* \leq \left[1 - \frac{4\mu\epsilon(1-\delta)^2}{\eta L}\left(\frac{1-\delta}{(1+\delta)^2} - \epsilon \right)\right] (f_k - f_*),
%\left(1 - \frac{4\mu\epsilon(1-\epsilon)}{\eta L}\right) (f_k - f_*)-(f_{k+1}-f_*) \ge \frac{2\epsilon(1-\epsilon)}{\eta (L-\mu)}\norm{\bfg_k + \mu(\bfx_* - \bfx_k)}^2,
\end{equation*}
which is valid for any noise level $\delta\in [0,1)$,  algorithm parameters  $\epsilon \in \left(0, \frac{1-\delta}{(1+\delta)^2}\right)$ and~{$\eta>1$}.
% are the parameters of the method.  %Apart from the Armijo rule, we also consider other inexact line search methods and noisy variants of GD.\\

\medskip
\noindent {\em Gradient descent  with Goldstein-terminated line search:}
	$$f_{k+1} - f_* \leq \left(1 - \frac{4\mu\epsilon(1-\delta)^2}{L}\left[\frac{1-\delta}{(1+\delta)^2} - (1-\epsilon)\right]\right) (f_k - f_*),$$
	which is valid for noise levels  $\delta \in [0,\sqrt{5}-2)$ and  algorithm parameter $\epsilon \in~\left(1-\frac{1-\delta}{(1+\delta)^2},\frac{1}{2} \right)$.

\medskip
\noindent {\em Gradient descent  with Wolfe-terminated line search:}
$$f_{k+1} - f_* \leq \left(1 - \frac{2\mu c_1(1-c_2)}{L}\right) (f_k - f_*),$$
 which is valid for any algorithm parameters $0 < c_1 < c_2 < 1$.
 
	 We show that the  bound for GD with  Armijo-terminated line search rule  is an improvement upon two existing bounds in the literature, see \cite[Proposition~3.3.5]{nemirovski} and \cite[Page~239]{ye}. On the other hand, 
our  results for GD with Goldstein  or Wolfe-terminated line search   are, to the best of our knowledge,~new.

The interested reader may find the code for numerically and symbolically verifying the results at \url{https://github.com/sandratsy/SumsOfSquares}.

\paragraph{Paper Organization.} The paper is organized as follows: Section~\ref{sec:SOS} introduces the SOS technique, explains how it is applied to derive worst-case bounds and describes the function class and algorithms we examine within this work. Furthermore, we determine a procedure for lifting degree-1 certificates from the univariate to the multivariate case and also prove the relation between PEP and SOS. In Section 3 we use the SOS framework to determine 
  new convergence result{s} for noisy gradient descent with   inexact line search methods (Armijo, Wolfe,  Goldstein). % and 
%Examples shown in Sections~\ref{sec:gd} and \ref{sec:pgm} work out the derivation of bounds for specific algorithms. Finally, %Section~\ref{sec:implementation} provides some implementation details, while
Lastly, Section~\ref{sec:conc} contains concluding remarks and suggests avenues for future work.

\paragraph{Note.} A preliminary version of this  paper was presented at the Signal Processing with Adaptive Sparse Structured Representations (SPARS) workshop in Toulouse, France in July 2019 \cite{spars}. Moreover, several additional convergence results obtained via the SOS approach including   GD with constant step size and  exact line search, and 
 proximal gradient with constant step size and  exact line search can be found in the M. Eng. thesis of the first author \cite{sandra}. These results have not been included in this manuscript as the exact same rates have been also derived  via the PEP framework, which as already discussed, is equivalent to degree-1 SOS proofs.  

\section{Description of our Approach}\label{sec:SOS}
\subsection{Background on Sum-of-Squares}
Before we provide the details of our approach, we need to introduce some necessary notation and definitions. 
For any $\bfa \in \N^n$, where $\N$ is the set of nonnegative integers, we denote by $\bfz^{\bfa}$ the monomial $z_1^{a_1}\dots z_n^{a_n}$.
The degree of the monomial $\bfz^{\bfa}$ is defined to be $|\bfa|=\sum_{i=1}^na_i$.
Let $\R[\bfz]_{n,d}$ denote the set of polynomials in $n$ variables $z_1,\ldots,z_n$, of degree at most $d$.
%Letting $m_d$ denote the set of monomials of degree at most $d$,
Any polynomial $p(\bfz) \in \R[\bfz]_{n,d}$ can be written as a linear combination of monomials of degree at most $d$, i.e., 
$	p(\bfz)=\sum_{|\bfa| \le d} p_{\bfa}\bfz^{\bfa}. $ %p(\bfz)=\sum_{\bfz^{\bfa}\in m_d} p_{\bfa}\bfz^{\bfa}.
An (even-degree) polynomial $p(\bfz)$ is called a {\em sum-of-squares} (SOS)
if there exist  polynomials $q_1(\bfz),\dots,q_m(\bfz) $ satisfying
$p(\bfz) = \sum_{i=1}^m q_i^2(\bfz). 
$ Note that if the degree of $p(\bfz)$ is equal to $2d$, all polynomials  $q_i(\bfz)$  will necessarily have degree at most $d$. 
It is instructive to think of the existence of an SOS decomposition as a tractable certificate for the global nonnegativity of $p(\bfz)$. Indeed, it is clear  that any SOS polynomial $p(\bfz)$ is also globally nonnegative, i.e., $p(\bfz)\ge 0$ for all $\bfz \in \R^n$. Furthermore, although less obvious, it is well-known that checking the existence of an SOS decomposition can be done efficiently using SDPs~\cite{SOStextbook}. 

%Concretely,  
%an even-degree polynomial $p(\bfz) = \sum_{\bfa} p_{\bfa} \bfz^{\bfa}$ admits an SOS decomposition where  $q_1,\ldots,q_m$ have degree $d$, %at most $d$ (for $d = {\rm deg}(p)/2$) i
%if and only if there exists a positive semidefinite (PSD) matrix $Q$ indexed by  $\{\bfa\in \mathbb{N}^n: |\bfa|\le d\}$~with%size $\begin{psmallmatrix}n+d \\ d\end{psmallmatrix}$ 
%
%\begin{equation}\label{eq:linear_constraints}
%	p_{\bfa} = \sum_{\bfb,\: \bfc \: : \: \bfb+\bfc\: = \:\bfa} Q_{\bfb,\bfc},\  \text{ for all }  |\bfa|\le d,
%\end{equation}
%where $Q_{\bfb,\bfc}$ denotes the element in the $(\bfb,\bfc)$ entry of $Q$. 

%This statement admits an easy proof: The polynomial $p(\bfz)$ is an SOS of degree-$d$ polynomials if and only if it can be expressed as a convex  quadratic form, i.e., $p(\bfz)=[\bfz]_d^\top Q [\bfz]_d$, where~$[\bfz]_d$ denotes the vector of monomials in $\bfz$ of degree at most $d$ and $Q$ is PSD. Equating %Matching 
%coefficients of all monomials in  $p(\bfz)=[\bfz]_d^\top Q [\bfz]_d$ gives
%%$\bfz$ gives %produces 
%the set of affine constraints given in \eqref{eq:linear_constraints}.
%Thus, deciding whether $p(\bfz)$ is an SOS is an SDP with matrix variables of size $\begin{psmallmatrix}n+d \\ d\end{psmallmatrix}$. %which for fixed $d$ is a polynomial in $n$.%can be solved in time polynomial in  $n^d$ \cite{SOStextbook}.

Moving beyond the problem of certifying global nonnegativity, a more general problem 
%of significant practical importance
is to certify the nonnegativity of a polynomial $p(\bfz)$ over a (basic) closed semi-algebraic~set 
$K = \left\{ \bfz\in \R^n :  \ h_i(\bfz) \ge 0,  \ i \in [m],  \  v_j(\bfz) = 0, \ j\in [m'] \right\},
$
i.e., to certify that $p(\bfz)\ge 0$ for all $\bfz\in K$.  Analogously to the case of global nonnegativity, we look for certificates that can be found efficiently using SDPs. One such choice are Putinar-type certificates \cite{putinar}, given~by:
\begin{equation}\label{eq:putinar}
p(\bfz) = \sigma_0(\bfz) + \sum_{i=1}^m \sigma_i(\bfz)h_i(\bfz) + \sum_{j=1}^{m'} \theta_j(\bfz)v_j(\bfz), %\text{ for all }  |\bfa|\le d.
\end{equation}
where the $\sigma_i$'s are themselves SOS polynomials and the $\theta_j$'s are arbitrary (i.e., not necessarily SOS) polynomials.
Clearly, the expression~\eqref{eq:putinar} serves as a certificate that $p(\bfz) \geq 0$ for all $\bfz \in K$ and moreover, the existence of such a representation (for a fixed degree~$d$) can be done using SDPs, e.g., see \cite{SOStextbook}.
%
% that  $p(\bfz)$ admits an SOS decomposition as in \eqref{eq:putinar} if and only if there exist $m+1$ PSD matrices $Q^0,Q^1,\dots,Q^m$~where
%\begin{equation}\label{eq:SOS_SDP_constraints}
%	p_{\bfa} = \sum_{\substack{\bfb,\: \bfc \: : \\ \bfb+\bfc\, =\,\bfa}} Q^0_{\bfb,\bfc} + \sum_{\substack{\bfb,\: \bfc,\: \bfd \: :\\ \bfb+\bfc+\bfd\,=\,\bfa}} \sum_{i=1}^m Q^i_{\bfb,\bfc} h_{i,\bfd} + \sum_{\substack{\bfa,\: \bfb \: : \\ \bfb+\bfc\,=\,\bfa}} \sum_{j=1}^{m'} \theta_{j,\bfb} v_{j,\bfc},
%\end{equation}
%where $h_i(\bfz)=\sum_{\bfd} h_{i,\bfd}\bfz^{\bfd}$ and $v_j(\bfz)=\sum_{\bfc}v_{j,\bfc}\bfz^{\bfc}$. %refers to the coefficient of the monomial $\bfz^{\bfd}$ in the polynomial $h_i$.
%Again, these affine constraints \eqref{eq:SOS_SDP_constraints} are obtained by expressing each SOS term in its convex quadratic form and matching coefficients of the monomials in $\bfz$.
%The problem of finding PSD matrices $Q^0,Q^1,\ldots, Q^m$ and unconstrained scalars $\theta_{j,\bfb}$ satisfying the affine constraints~\eqref{eq:SOS_SDP_constraints} is a feasibility SDP.

%Lastly, we note that Putinar-type certificates are also necessary for strict nonnegativity over a semi-algebraic set under fairly general assumptions, which are, for  example, satisfied when the underlying set $K$ is compact (e.g., see \cite{putinar} or \cite[Theorem 3.20]{SOS-Laurent}). On the negative side, the only general bounds on the degrees of the SOS certificates in \eqref{eq:putinar} are exponential in the number of variables of the polynomial.

\subsection{Algorithm Analysis Using SOS Certificates}
%As mentioned in the introduction, we focus on (monotone) first-order algorithms and use as performance metric the objective function accuracy. 
Fixing a family of functions $\F$ and a first-order algorithm $\A$---one that uses only gradient information---our goal is to find the best (smallest) contraction factor $t\in (0,1)$ that is valid over all functions in $\F$ and all sequences of iterates that can be generated using the algorithm~$\A$. Concretely, for any fixed $k$, we want to estimate the minimum $t\in (0,1)$ satisfying $f_{k+1} - f_*\le t(f_k - f_*),$ for all $f\in \F$ and $\bfx_{k+1}=\A\left(\bfx_k, f_k, \bfg_k\right)$.
We address this question using SOS certificates. To employ an SOS approach, we first need to identify polynomial inequalities $h_i(\bfz)\ge 0$ and polynomial equalities $v_j(\bfz)=0$ in the variables 
\begin{equation}\label{eq:zvar}
\bfz:= (f_*, f_k, f_{k+1},\bfx_*,\bfx_k, \bfx_{k+1}, \bfg_*, \bfg_k, \bfg_{k+1})\in \R^{6n+3}
\end{equation}
that should be necessarily satisfied following the choice of the class of functions $\F$ and the first-order algorithm $\A$. Setting $K$ to be the semi-algebraic set
defined by the identified polynomial equalities and inequalities, i.e.,
\begin{equation}
K:=\left\{\bfz: h_i(\bfz)\ge 0, \ i\in [m],  \quad v_j(\bfz)=0, \ j\in [m'] \right\}, \notag
\end{equation}
%\begin{equation}
%K := \left\{ \bfz : h_i(\bfz) \ge 0,  \ 1\leq i\leq m,  \quad  g_j(\bfz) = 0, \ m+1\leq j\leq m' \right\},
%\end{equation}
it follows immediately that if the polynomial
\begin{equation}\notag
p_t(\bfz) := t(f_k - f_*) - (f_{k+1} - f_*) %removed comma; the sentence doesn't require a comma here even if the equation were inline
\end{equation}
is nonnegative over the set $K$ for some $t\in (0,1)$, then $t$ also serves as an upper bound on the worst-case rate $t_*$, or, in other words, $t_*$ is upper bounded by the value of the following optimization problem
\begin{equation}\label{eq:estimation}
t_{\poly}:= \inf \{ t: \ p_t(\bfz)\ge 0 \ \forall \bfz\in K,\  t\in (0,1)\},
\end{equation} 
where the decision variable is the scalar $t$. 
As the optimization problem \eqref{eq:estimation} involves a polynomial nonnegativity constraint (over a semi-algebraic set) it is in general hard---in fact, strongly NP-hard \cite{ahmadi}. To obtain tractable upper bounds, we replace the  constraint that $p_t(\bfz)$ is nonnegative over $K$ by asking that $p_t(\bfz)$ admits   an  SOS certificate of the form   \eqref{eq:putinar}, which clearly certifies nonnegativity over $K$.  Concretely, for any $d\ge 0$ and $n\geq 1$, we get the~SDP:
 \begin{equation}\label{eq:sdp}
\begin{aligned}t_d :=\text{minimize} \ \ & t\\
\text{subject to} \ \  & p_t(\bfz)=s_0(\bfz) + \sum_{i=1}^m \sigma_i(\bfz)h_i(\bfz) + \sum_{j=1}^{m'} \theta_j(\bfz)v_j(\bfz),\\
& t \in (0,1),\\
& \sigma_0(\bfz): \text{SOS polynomial with } \deg (\sigma_0(\bfz))\le 2d, \\
& \sigma_i(\bfz): \text{SOS polynomial with } \deg (\sigma_i(\bfz)h_i(\bfz))\le 2d, \\
& \theta_j(\bfz): \text{arbitrary polynomial with } \deg (\theta_j(\bfz)v_j(\bfz))\le 2d, 
\end{aligned}
\end{equation}
where $\bfz \in \R^{6n+3}$.
For any fixed integer $d\ge 0$ and $n\geq 1$, the optimization problem~\eqref{eq:sdp} is an SDP, and consequently, it can be solved  in polynomial-time to any desired accuracy.
%We now claim, and will later prove, that $t_d := t_{d,1} \geq t_{d,2} \geq \dots$ for all $n \geq 1$, under certain conditions on $p_t(\bfz)$, the $h_i(\bfz)$'s and the $g_j(\bfz)$'s.
Furthermore, for a fixed $n$, it follows immediately from the definitions~that
\begin{equation*}
	t_*\le t_{\poly}\le ...\le t_{d+1}\le t_d, \text{ for all } d \in \mathbb{N}. %ge 0.
\end{equation*}
In other words, as $d$ increases, the  SDPs given in \eqref{eq:sdp} give increasingly better---more precisely, no worse---upper bounds on the worst-case ratio $t_*$. On the negative side, the sizes of these SDPs grow as $\mathcal{O}(n^d),$ so in practice, working with large values of $d$ is computationally prohibitive. 
Summarizing, our strategy for estimating the  worst-case rate  consists of the following~steps:
%Our goal is to find the minimum $t$ such that $p(\bfz)$ can be expressed as a SOS, guaranteeing nonnegativity.
%We seek to find the minimum $t \in (0,1)$ for which the SDP in \eqref{eq:SOS_SDP_constraints} is feasible.
%Overall, for a specific algorithm and function class, we perform the following
\begin{enumerate}
	\item Identify polynomial inequality and equality constraints ${h_i(\bfz) \geq 0}, \  v_j(\bfz) =~0$ in the variable $\bfz$ (recall \eqref{eq:zvar}) that are implied by choosing a function class and an algorithm.  
	\item Fix a degree $d\in \N$ for the SOS certificate, i.e., for the degrees of the polynomials $\sigma_i$'s and $\theta_j$'s. 
	Higher degree certificates  allow for tighter bounds but are more difficult to find due to the increase in size of the~SDP.
	\item Numerically solve the SDP in \eqref{eq:sdp} using degree-$d$ SOS certificates multiple times, varying the parameters corresponding to the algorithm and the function class.
	%any parameters (of the algorithm or function class). 
	This allows us to ``guess'' the analytic form of the optimal variables for \eqref{eq:sdp}.
	\item Lastly, we verify that the identified solution from step 3 is indeed feasible for \eqref{eq:sdp}. Determining feasibility gives  an analytic upper bound on the best contraction factor $t_d$ that can certified using degree-$d$ SOS certificates.
	%Once verified, the optimal solution of \eqref{eq:sdp}, denoted by $t_d$, is the best contraction factor that can certified using degree-$d$ SOS certificates.
\end{enumerate}

\paragraph{Implementation Details.}
Throughout this paper, we restrict our attention to {degree-1} SOS certificates, as our main goal is to derive  rates symbolically (see Section~\ref{sec:choices}). The derivation of the affine constraints defining the feasible region of the SDP \eqref{eq:sdp} was done 
%by hand for one example in Section \ref{sec:eg1}, 
by matching coefficients in \eqref{eq:putinar}.
%While this process was tedious, it allowed for more flexibility in constructing the SDPs; see e.g., Section \ref{sec:GD_const} where we 
%constructed a second SDP \eqref{eq:GD_SDP2} out of \eqref{eq:sdp} by 
%adding an additional constraint that enforces the sparsity of the SOS decomposition. 
The SDP \eqref{eq:sdp} was solved with CVX \cite{cvx1,cvx2}, using the supported SDP solver SDPT3 \cite{SDPT3,SDPT3-2}. %Doing so by hand is tedious and also impractical for higher-degree certificates.
Fortunately, there are many SOS optimization toolboxes such as YALMIP that automate the process of matching coefficients and constructing the SDP. 
%The user only needs to input the constraints and set the degree of the certificate he/she wishes to search for. 
%The remainder of our work was hence implemented in YALMIP.
%We have implementations of our current work in YALMIP as well, in order to numerically verify the correctness of the derivations done by hand.
%If the flexibility in SDP construction is not required,
%e.g., if only the numerical value of the optimal contraction factor $t_*$ in \eqref{eq:sdp} is needed,
%an SDP solver like YALMIP \cite{yalmip} could be used instead.
%YALMIP automates the process of matching coefficients and constructing the corresponding SDP.
%One notable difference between CVX and YALMIP is that the latter can only check for the existence of an SOS decomposition. In other words, YALMIP can check whether or not \eqref{eq:sdp} is feasible for a fixed $t$, but not minimize over $t$ directly.
%This issue is easily circumvented by performing bisection in the interval $(0,1)$ for $t$. 
Finally, verification of the identified solution was done through MATLAB's Symbolic Math Toolbox and Mathematica \cite{Mathematica}. Mathematica was used to first verify that the optimal  matrices  are PSD, before we found their corresponding SOS decompositions analytically.
For the interested reader, the codes for implementation of the SDPs and verification of the solutions may be found at \url{https://github.com/sandratsy/SumsOfSquares}.

\subsection{Choices Specific to this Work}\label{sec:choices}

\paragraph{Function classes of interest.} Consider parameters $0\le\mu < L < +\infty$.
%\footnote{I think allowing $L=\infty$ is quite nonstandard. Its best to set $L<\infty$ and then add that ``Extending this notation, we denote by $\mathcal{F}_{0,\infty}$ the class of closed, proper, and convex functions. Then, we use this notation instead of $(0,\infty)$-smooth} 
In this work, we only consider the class of \emph{$L$-smooth, $\mu$-strongly convex}  functions---also known as {\em $(\mu,L)$-smooth functions}---with domain $\R^n$, which we denote by $\FmuLRn$. 
Recall that a proper, closed, convex function ${f:\R^n \rightarrow \R \cup \{+\infty\}}$ is called {\em $L$-smooth} if  
\begin{equation*}
	\norm{\bfg_1 - \bfg_2} \leq L\norm{\bfx_1 - \bfx_2},  \ \forall \bfx_1, \bfx_2 \in \R^n, \ \bfg_1= \nabla f(\bfx_1), \bfg_2= \nabla f(\bfx_2),
\end{equation*}
and  {\em $\mu$-strongly convex} if the function $	f(\bfx) - \frac{\mu}{2}\norm{\bfx}^2 $  is convex, where $\|\cdot \|$ denotes the usual Euclidean norm. %Extending this notation, we denote by $\F_{0,\infty}$ the class of proper, closed and convex functions.% We say that  $\FmuLRn$  is the class of {\em $(\mu,L)$-smooth functions}. 

Throughout this work, we  use the following set  of necessary and sufficient conditions developed in \cite{Taylor} for the existence of a function in $\FmuLRn$ generating data triples $\{(\bfx_i, f_i, \bfg_i)\}_{i \in I}$. %Such a set of data triples is termed \emph{$\FmuL$-interpolable}.

\begin{theorem}\label{thm:interpolability}
	Given a set $\{(\bfx_i, f_i, \bfg_i)\}_{i \in I}$, there exists $f \in \FmuLRn$ where $f_i = f(\bfx_i)$ and $\bfg_i =\nabla f(\bfx_i)$ for all $i\in I$,  if and only~if,	for all $i\ne j\in I$:
	{\small \begin{equation*}f_i - f_j - \bfg_j^\top (\bfx_i - \bfx_j)  \geq \frac{L}{2(L - \mu)}\left( \frac{1}{L}\norm{\bfg_i - \bfg_j}^2  + \mu\norm{\bfx_i - \bfx_j}^2 - 2\frac{\mu}{L}(\bfg_j - \bfg_i)^\top (\bfx_j - \bfx_i) \right). \end{equation*}}
\end{theorem}
Applying Theorem \ref{thm:interpolability} to the data triples $(\bfx_k,f_k,\bfg_k)$, $(\bfx_{k+1},f_{k+1},\bfg_{k+1})$, and $(\bfx_*,f_*, \bfg_*)$ we get    six polynomial constraints that we denote throughout this paper by 
$	h_1(\bfz) \geq 0, \ldots, h_6(\bfz) \geq 0. $
Specifically,   setting
%that $\bfg_*=0$ and also
  $\alpha:=\frac{1}{2(1 - \mu / L)}$, the six $\FmuL$-interpolability conditions are:
%\begin{equation}
%\begin{aligned}\label{eq:interpolability}
%&f_k - f_{k+1} - \bfg_{k+1}^\top(\bfx_k - \bfx_{k+1}) - \alpha\left[ \frac{1}{L}\norm{\bfg_k-\bfg_{k+1}}^2 + \mu\norm{\bfx_k-\bfx_{k+1}}^2 \right.\\
%&\quad \left. {} - 2\frac{\mu}{L}(\bfg_{k+1}-\bfg_k)^\top(\bfx_{k+1}-\bfx_k) \right] \ge 0\\
%&f_k - f_* - \bfg_*^\top(\bfx_k - \bfx_*) - \alpha\left[ \frac{1}{L}\norm{\bfg_k - \bfg_*}^2 + \mu\norm{\bfx_k-\bfx_*}^2 \right.\\
%&\quad \left. {} - 2\frac{\mu}{L}(\bfg_*-\bfg_k)^\top(\bfx_*-\bfx_k) \right]\ge 0\\
%&f_{k+1} - f_k - \bfg_k^\top (\bfx_{k+1} - \bfx_k) - \alpha\left[\frac{1}{L}\norm{\bfg_{k+1} - \bfg_k}^2 + \mu\norm{\bfx_{k+1} - \bfx_k}^2 \right.\\
%&\quad \left. {} - 2\frac{\mu}{L}(\bfg_k - \bfg_{k+1})^\top (\bfx_k - \bfx_{k+1}) \right]\ge 0 \\
%&f_{k+1} - f_* - \bfg_*^\top(\bfx_{k+1} - \bfx_*) - \alpha\left[ \frac{1}{L}\norm{\bfg_{k+1} - \bfg_*}^2 + \mu\norm{\bfx_{k+1} - \bfx_*}^2 \right. \\
%&\quad \left. {} - 2\frac{\mu}{L}(\bfg_*-\bfg_{k+1})^\top (\bfx_* - \bfx_{k+1}) \right] \ge 0 \\
%&f_* - f_k - \bfg_k^\top (\bfx_* - \bfx_k) - \alpha\left[ \frac{1}{L}\norm{\bfg_* - \bfg_k}^2 + \mu\norm{\bfx_* - \bfx_k}^2 \right.\\
%&\quad \left. {} - 2\frac{\mu}{L}(\bfg_k-\bfg_*)^\top (\bfx_k - \bfx_*) \right]\ge 0 \\
%&f_* - f_{k+1} - \bfg_{k+1}^\top (\bfx_* - \bfx_{k+1}) - \alpha\left[ \frac{1}{L}\norm{\bfg_* - \bfg_{k+1}}^2 + \mu\norm{\bfx_* - \bfx_{k+1}}^2 \right.\\
%&\quad \left.{} - 2\frac{\mu}{L}(\bfg_{k+1}-\bfg_*)^\top (\bfx_{k+1} - \bfx_*) \right]\ge 0,
%\end{aligned}
%\end{equation}
{\small
\begin{align}
&f_k\! -\!  f_{k+1} - \bfg_{k+1}^\top(\bfx_k - \bfx_{k+1}) - \alpha\left( \frac{1}{L}\norm{\bfg_k-\bfg_{k+1}}^2 + \mu\norm{\bfx_k-\bfx_{k+1}}^2  
\! -\!  2\frac{\mu}{L}(\bfg_{k+1}\! -\! \bfg_k)^\top(\bfx_{k+1}\! -\! \bfx_k) \right) \! \ge\!  0 \notag \\
&f_k - f_* - \bfg_*^\top(\bfx_k - \bfx_*) - \alpha\left( \frac{1}{L}\norm{\bfg_k - \bfg_*}^2 + \mu\norm{\bfx_k-\bfx_*}^2  - 2\frac{\mu}{L}(\bfg_*-\bfg_k)^\top(\bfx_*-\bfx_k) \right)\ge 0 \notag \\
&f_{k+1} - f_k - \bfg_k^\top (\bfx_{k+1} - \bfx_k) - \alpha\left(\frac{1}{L}\norm{\bfg_{k+1} - \bfg_k}^2 + \mu\norm{\bfx_{k+1} - \bfx_k}^2 - 2\frac{\mu}{L}(\bfg_k - \bfg_{k+1})^\top (\bfx_k - \bfx_{k+1}) \right)\ge 0 \notag \\
&f_{k+1} - f_* - \bfg_*^\top(\bfx_{k+1} - \bfx_*) - \alpha\left(\frac{1}{L}\norm{\bfg_{k+1} - \bfg_*}^2 + \mu\norm{\bfx_{k+1} - \bfx_*}^2  - 2\frac{\mu}{L}(\bfg_*-\bfg_{k+1})^\top (\bfx_* - \bfx_{k+1}) \right) \ge 0 \notag \\
&f_* - f_k - \bfg_k^\top (\bfx_* - \bfx_k) - \alpha\left( \frac{1}{L}\norm{\bfg_* - \bfg_k}^2 + \mu\norm{\bfx_* - \bfx_k}^2 - 2\frac{\mu}{L}(\bfg_k-\bfg_*)^\top (\bfx_k - \bfx_*) \right)\ge 0 \notag \\
&f_* \! -\!  f_{k+1} - \bfg_{k+1}^\top (\bfx_* - \bfx_{k+1}) - \alpha\left(\frac{1}{L}\norm{\bfg_* - \bfg_{k+1}}^2 + \mu\norm{\bfx_* - \bfx_{k+1}}^2\!  -\!  2\frac{\mu}{L}(\bfg_{k+1}\! -\! \bfg_*)^\top (\bfx_{k+1} \! -\!  \bfx_*) \right)\! \ge\! 0. \label{eq:interpolability}
\end{align}}

\subsection{Lifting Univariate Certificates}
As already mentioned,   we restrict our attention to degree-1 SOS certificates (recall~\eqref{eq:sdp}). In this setting, $\sigma_0(\bfz)$ is an SOS of linear polynomials and, since the polynomials $h_i(\bfz)$ and $v_j(\bfz)$ we consider are degree-2, the $\sigma_i(\bfz)$'s  need to  be degree-0 SOS polynomials and the $\theta_j(\bfz)$'s  degree-0 %arbitary 
polynomials.
%we set polynomials $\sigma_i(\bfz)$ to be degree-0 SOS polynomials and $\theta_j(\bfz)$ to be degree-0 %arbitary 
 The SOS certificate can be thus expressed~as:
\begin{equation}\label{eq:putinardeg1}
	p(\bfz) = \sigma_0(\bfz) + \sum_{i=1}^m \sigma_i h_i(\bfz) + \sum_{j=1}^{m'} \theta_j v_j(\bfz),
\end{equation}
where $\sigma_i \in \R_+$ and $\theta_j \in \R$.
We claim that the form of the polynomials $p_t(\bfz)$, $h_i(\bfz)$'s and $v_j(\bfz)$'s, combined with the specific choice of SOS certificates under consideration (i.e., degree-1 certificates) allow us to only consider the univariate case $n=1$. Concretely, we show in the rest of this section that an SOS certificate for some contraction factor  $t\in (0,1)$ in the univariate case, induces  an  SOS certificate  for  the same contraction factor in  the multivariate case ($n>1$).
To see this, first we rearrange the variable $\bfz= (f_*, f_k, f_{k+1},\bfx_*,\bfx_k, \bfx_{k+1}, \bfg_*, \bfg_k, \bfg_{k+1})$ as $\bfz=(\bfz_0, \bfz_1, \ldots, \bfz_n),$~where
\begin{equation}\label{eq:zvar2}
	\bfz_0=  (f_*, f_k, f_{k+1}) \ \text{ and }\  \bfz_\ell = (x_*(\ell),x_k(\ell), x_{k+1}(\ell), g_*(\ell), g_k(\ell), g_{k+1}(\ell)), \ \ell\in [n]
\end{equation}
where  $x_*(\ell)$ denotes  the $\ell$\textsuperscript{th} coordinate of $\bfx_*$ for $\ell\in [n]$. 
%and so on. 

\begin{theorem}\label{lifting}
Assume that the  performance measure polynomial and the constraint functions   are separable with respect to the blocks of variables $\bfz_0, \bfz_1,\ldots,\bfz_n$, they are invariant with respect to permutations of the blocks of variables $\bfz_1,\ldots,\bfz_n$, and that they  have no  constant terms.
Then, a degree-1 SOS certificate for a rate $t\in (0,1)$ in the univariate case (i.e., $n=1$) can be lifted to degree-1 certificate for the general case (i.e., $n>1$).

\end{theorem}

Note that the structural assumptions on the performance measure polynomial and the constraint functions imply that they have the form
\begin{equation}\label{polynomials}
\begin{aligned}
p_t(\bfz) &= p_t^0(\bfz_0) + \sum_{\ell=1}^n p_t^1(\bfz_\ell) \\
h_i(\bfz) &= h_i^0(\bfz_0) + \sum_{\ell=1}^n h_i^1(\bfz_\ell) \\
v_j(\bfz) &= v_j^0(\bfz_0) + \sum_{\ell=1}^n v_j^1(\bfz_\ell),
\end{aligned}
\end{equation}
for some polynomials $p_t^0, p_t^1, h_i^0, h_i^1, v_j^0$ and $v_j^1$.
% i.e.,  the polynomials $p_t, h_i,$ and $v_j $  have the following three  properties: First, they

Furthermore, note that all the performance measure polynomials  (e.g., $t(f_k - f_*)- (f_{k+1} - f_*)$ and $t\norm{\bfx_k-\bfx_*}^2 - \norm{\bfx_{k+1} - \bfx_*}^2$) and the constraint functions encountered thus far have the form~\eqref{polynomials}.
%\begin{equation}\label{polynomials}
%\begin{aligned}
%	p_t(\bfz) &= \begin{cases}
%	p_t^0(\bfz_0) + p_t^1(\bfz_1) & n = 1\\
%	  p_t^0(\bfz_0) + \sum_{\ell=1}^n p_t^1(\bfz_\ell) & n > 1
%	\end{cases}\\
%	h_i(\bfz) &= \begin{cases}
%  h_i^0(\bfz_0) + h_i^1(\bfz_1) & n = 1\\
%	 h_i^0(\bfz_0) + \sum_{\ell=1}^n h_i^1(\bfz_\ell) & n > 1
%	\end{cases}\\
%	g_j(\bfz) &= \begin{cases}
%	 g_j^0(\bfz_0) + g_j^1(\bfz_1) & n = 1\\
%	 g_j^0(\bfz_0) + \sum_{\ell=1}^n g_j^1(\bfz_\ell) & n > 1,
%	\end{cases}
%\end{aligned}
%\end{equation}
 Furthermore,  \eqref{polynomials} is satisfied when the polynomial constraints take the form given in \eqref{const3}, i.e, the constraints are linear in the $f$'s and in the inner products of the $\bfx_i$'s and $\bfg_i$'s.

\begin{proof}
To prove the theorem, note that an SOS  certificate $\left\{Q, \{\sigma_i\}_i, \{\theta_j \}_j \right\}$, (i.e.,  $Q$ is a PSD matrix, $\{\sigma_i \}_i \subseteq \R_+$ and $\{\theta_j \}_j \subseteq \R$) for a rate  $t\in (0,1)$ in the general case $n>1$ has the following form:
\begin{equation}\label{sdvdfgbg}
p_t(\bfz)=\begin{pmatrix}1\\ \bfz_0\\ \bfz_1\\ \vdots\\ \bfz_n\end{pmatrix}^\top Q\begin{pmatrix} 1\\\bfz_0\\ \bfz_1\\ \vdots\\\bfz_n\end{pmatrix}+\sum_{i=1}^m \sigma_i\left( h^0_i(\bfz_0)+\sum_{\ell=1}^n h^1_i(\bfz_\ell)\right)+\sum_{j=1}^{m'} \theta_j \left( v_j^0(\bfz_0) + \sum_{\ell=1}^n v_j^1(\bfz_\ell)  \right).
\end{equation}
As the polynomials have no constant terms, it follows immediately that $Q_{11}=~0$. Furthermore, as the polynomials are separable with respect to the blocks of variables $(\bfz_0, \bfz_1, \ldots, \bfz_n),$ $Q$ is block diagonal. Using these two observations, \eqref{polynomials} and \eqref{sdvdfgbg} imply~that:
\begin{align}
	p_t^0(\bfz_0) &=  \bfz_0^\top Q_0 \bfz_0 + \sum_{i=1}^m \sigma_i  h_i^0(\bfz_0) + \sum_{j=1}^{m'} \theta_j  v_j^0(\bfz_0), \label{eq:mult1}\\
	p_t^1(\bfz_\ell) &= \bfz_\ell^\top Q_\ell\bfz_\ell + \sum_{i=1}^m \sigma_i h_i^1(\bfz_\ell) + \sum_{j=1}^{m'} \theta_j v_j^1(\bfz_\ell), \quad  \ell\in [n]. \label{eq:mult2}
\end{align}
%with respect to the blocks $1, \bfz_0, \bfz_1,\ldots,\bfz_n$. As a consequence, \eqref{sdvdfgbg} is equivalent to
%\begin{equation}
%p_t(\bfz)=\sum_{\ell=1}^n[\bfz_\ell]^\top Q_\ell[\bfz_\ell]+\sum_{i=1}^m\lambda_i\left(h^0_i(\bfz_0)+\sum_{\ell=1}^nh^u_i(\bfz_\ell)\right)+\sum_{j=1}^{m'}\lambda_j\sum_{\ell=1}^ng_j^u(\bfz_\ell),
%\end{equation}
%and rearranging terms,  this gives:
%\be\label{cdvfgbfhnh}
%p_t(\bfz_0)=\sum_{\ell=1}^n\left([\bfz_\ell]^\top Q_\ell[\bfz_\ell]+ \sum_{i=1}^m\lambda_ih^u_i(\bfz_\ell)+\sum_{j=1}^{m'}\lambda_jg_j^u(\bfz_\ell)\right)+   \sum_{i=1}^m\lambda_i h^0_i(\bfz_0). 
%\ee
Lastly, assume there exists an SOS certificate for a rate $t\in (0,1)$ in the univariate case, i.e., a PSD matrix  $\tilde{Q}$ and scalars $ \{\tilde{\sigma}_i\}_i\subseteq \R_+, \{ \tilde{\theta}_j \}_j \subseteq \R$~where 
\begin{equation}
	p_t(\bfz)=\begin{pmatrix}1\\ \bfz_0\\ \bfz_1\end{pmatrix}^\top \tilde{Q}\begin{pmatrix} 1\\\bfz_0\\ \bfz_1\end{pmatrix}+\sum_{i=1}^m \tilde{\sigma}_i\left( h^0_i(\bfz_0)+ h^1_i(\bfz_1)\right)+\sum_{j=1}^{m'} \tilde{\theta}_j \left(v_j^0(\bfz_0) + v_j^1(\bfz_1) \right). \notag
\end{equation}
As before, this  may be decomposed into
\begin{align}
	p_t^0(\bfz_0) &= \bfz_0^\top \tilde{Q}_0 \bfz_0 + \sum_{i=1}^m \tilde{\sigma}_i h_i^0(\bfz_0) + \sum_{j=1}^{m'} \tilde{\theta}_j v_j^0(\bfz_0)\label{eq:uni1}\\
	p_t^1(\bfz_1) &= \bfz_1^\top \tilde{Q}_1 \bfz_1 + \sum_{i=1}^m \tilde{\sigma}_i h_i^1(\bfz_1) + \sum_{j=1}^{m'} \tilde{\theta}_j v_j^1(\bfz_1). \label{eq:uni2}
\end{align}
Comparing equation \eqref{eq:mult1} with \eqref{eq:uni1} and equation \eqref{eq:mult2} with \eqref{eq:uni2}, we see that
$Q_0 = \tilde{Q}_0$, $Q_\ell = \tilde{Q}_1,  \ \ell \in [n]$, $\sigma_i = \tilde{\sigma}_i, \ i \in [m]$, $\theta_j = \tilde{\theta}_j, \ j \in [m']$
is a valid certificate for the multivariate case. 
\end{proof}

Lastly, we note that  for higher-degree SOS certificates (beyond degree-1), it  is not immediately apparent how to %we have not been able to 
verify that a certificate for the univariate case induces one for the multivariate case.

\subsection{Dual of the SOS Hierarchy}\label{sec:duality}%\footnote{notation issue: 0 vs 1 or k vs k+1}
In this section we determine the exact relationship between the PEP and the SOS hierarchy introduced in this work. Specifically, we show that:
\begin{theorem}\label{sosvspep}If  the functional and algorithmic constraints are linearly Gram-representable  (i.e., \eqref{const1} holds),
the 1-step PEP applied to a contractive algorithm is equivalent to the first-level of the SOS hierarchy.
\end{theorem}

\begin{proof} For concreteness, we consider the 1-step PEP (i.e.,  where we only take~1 step using algorithm $\mathcal{A}$) with respect to the performance metric given by the objective function accuracy. Similar arguments apply when the performance is measured using the distance from optimality or the residual gradient~norm. 

The corresponding optimization problem is given by:
\begin{equation*}
\begin{aligned} 
	\underset{f,\bfx_0,\bfx_1,\bfx_*}{\text{maximize}} \ \ &  f(\bfx_1)-f(\bfx_*)\\
	\text{subject to} \ \ &f\in \F, \\
		& \bfx_1=\A\left( \bfx_0, f_0, \nabla f(\bfx_0) \right), \\
		& \bfg_* = 0, \ f(\bfx_0)-f(\bfx_*)\le R,\\
		& \bfx_0,\bfx_1, \bfx_*\in \R^n.
\end{aligned}
\end{equation*}
If the $\F$-interpolability and algorithmic conditions are of the form given in \eqref{const1} with $a_i = b_i = 0$, the equivalent \eqref{f-PEP} may be relaxed into an SDP of the following form:%\footnote{shouldn't we also include here $\bfg_*=0?$. Also need to define $I$}
\begin{equation}\label{expl:1-step-sdp}
\begin{aligned} 
\underset{\{\bfx_i,\bfg_i,f_i\}_{i\in I}}{\text{maximize}} \ \ &  f_1-f_*\\
\text{subject to} \ \ & \inn{c_i}{\bff} + \inn{C_i}{G} \geq 0, \ i=1,\dots,m, \\
					& \inn{d_j}{\bff} + \inn{D_j}{G} = 0, \ j=1,\dots,m', \\
					& f_0 - f_* \leq R,
\end{aligned}
\end{equation}
where we recall that $\bff = (f_0,f_1,f_*)$, $X = (\bfx_0 \ \bfx_1 \ \bfx_* \ \bfg_0 \ \bfg_1 \ \bfg_*)\in \R^{n\times 6}$, $G = X^\top X$ and $I = \{0,1,*\}$.
Note that $f_1-f_*$ may be expressed as $\inn{(0,1,-1)}{\bff}$ and $f_0-f_*$ as $\inn{(1,0,-1)}{\bff}$. Setting $\sigma_i$, $\theta_j$ and $t$ to be the Lagrange multipliers of the three sets of constraints in \eqref{expl:1-step-sdp} respectively, the dual of \eqref{expl:1-step-sdp} is
\begin{equation}\label{expl:pep-dual}
\begin{aligned}
\underset{t,\{\sigma_i\}_{i=1}^m, \{\theta_j\}_{j=1}^{m'}}{\text{minimize}} \ \ & tR \\
\text{subject to} \ \ & \begin{pmatrix}0 \\ 1 \\ -1\end{pmatrix}-t\begin{pmatrix}1 \\ 0 \\-1\end{pmatrix} + \sum_{i=1}^m \sigma_i c_i + \sum_{j=1}^{m'} \theta_j d_j = 0\\
					  & \sum_{i=1}^m \sigma_i C_i + \sum_{j=1}^{m'} \theta_j D_j \preceq 0 \\
					  & t \geq 0, \ \sigma_i \geq 0, \ \theta_j \in \R.
\end{aligned}
\end{equation}
On the other hand, using the SOS approach and restricting our attention to degree-1 certificates, the SOS-SDP defined in  \eqref{eq:sdp} is given by:
\begin{equation*}
\begin{aligned}
	\underset{t,Q,\{\sigma_i\}_{i=1}^m, \{\theta_j\}_{j=1}^{m'}}{\text{minimize}} \ \ & t \\
	\text{subject to} \ \ \ \ & t(f_k - f_*) - (f_{k+1} - f_*) = \begin{pmatrix}1 \\ \bfz \end{pmatrix}^\top Q \begin{pmatrix}1 \\ \bfz \end{pmatrix} + \sum_{i=1}^m \sigma_i h_i(\bfz) + \sum_{j=1}^{m'} \theta_j v_j(\bfz), \\
						& t\in (0,1),\\
						& Q \succeq 0, \ \sigma_i \geq 0, \ \theta \in \R,
\end{aligned}
\end{equation*}
where  we define  as before  (recall \eqref{eq:zvar2})
%$\bfz$ 
%$\bfz= (f_*, f_k, f_{k+1},\bfx_*,\bfx_k, \bfx_{k+1}, \bfg_*, \bfg_k, \bfg_{k+1})$ as 
$\bfz=(\bfz_0, \bfz_1, \ldots, \bfz_n),$ with  	$\bfz_0=  ( f_k, f_{k+1}, f_*)$ and 
\begin{equation*}
	\bfz_\ell = (x_*(\ell),x_k(\ell), x_{k+1}(\ell), g_*(\ell), g_k(\ell), g_{k+1}(\ell)), \ \ell\in [n],
\end{equation*}
 and furthermore, 
%is defined in \eqref{eq:zvar}   
  the polynomial constraints $h_i(\bfz) \geq 0$ and $v_j(\bfz) = 0$ have the form given in  \eqref{const3} with $a_i = b_i = 0$, i.e., 
\begin{equation}\label{sdvrgbgr}
	h_i(\bfz)=\inn{c_i}{\bfz_0} + \sum_{\ell=1}^n \bfz_\ell^\top C_i \bfz_\ell   \text{ and } v_i(\bfz)= \inn{d_j}{\bfz_0} + \sum_\ell \bfz_\ell^\top D_j \bfz_\ell.
\end{equation}
Note that polynomials of the form  \eqref{sdvrgbgr} satisfy the requirement identified in~\eqref{polynomials}. In particular, this implies that the matrix $Q$ is block-diagonal. %with blocks 
%$Q_0, Q_1, \dots, Q_n$, with $Q_1 = \ldots = Q_n$.
Based on this, 
the constraint 
\begin{equation*}t(f_k - f_*) - (f_{k+1} - f_*) = \begin{pmatrix}1 \\ \bfz \end{pmatrix}^\top Q \begin{pmatrix}1 \\ \bfz \end{pmatrix} + \sum_{i=1}^m \sigma_i h_i(\bfz) + \sum_{j=1}^{m'} \theta_j v_j(\bfz),\end{equation*}
is equivalent to the equality of the following polynomials in the variables $\bfz=(\bfz_0, \bfz_1, \ldots, \bfz_n)$:
\begin{equation*}
\begin{aligned} 
&0=  \bfz_0^\top \left( t\begin{pmatrix}1 \\ 0 \\-1\end{pmatrix}-\begin{pmatrix}0 \\ 1 \\-1\end{pmatrix}-  \sum_{i=1}^m\sigma_i{c_i} - \sum_{j=1}^{m'} \theta_j {d_j}\right),\\
	& 0  =  \bfz_\ell^\top \left( Q_\ell  + \sum_{i=1}^m \sigma_i  {C_i}   + \sum_{j=1}^{m'} \theta_j {D_j}\right)\bfz_\ell, \ \forall \ell,
\end{aligned}
\end{equation*}
which is in turn equivalent to:
\begin{equation*}\begin{aligned} 
&0=   t\begin{pmatrix}1 \\ 0 \\-1\end{pmatrix}-\begin{pmatrix}0 \\ 1 \\-1\end{pmatrix}-  \sum_{i=1}^m\sigma_i{c_i} - \sum_{j=1}^{m'} \theta_j {d_j},\\
	& 0  =  Q_{\ell}  + \sum_{i=1}^m \sigma_i  {C_i}   + \sum_{j=1}^{m'} \theta_j {D_j}, \ \forall \ell.
		\end{aligned}
		\end{equation*}
		  Thus, the SOS-SDP may be expressed as:
%\red{When searching for degree-1 certificates, we may consider the univariate case without loss of generality. Since our polynomial constraints are of the form given in \eqref{const3}, which is equivalent to \eqref{const1} with $a_i = b_i = 0$,} the SOS-SDP may be expressed as:
%\begin{equation}\label{expl:sos_pre_final}
%\begin{aligned}
%	\underset{t,Q,\{\sigma_i\}_{i=1}^m, \{\theta_j\}_{j=1}^{m'}}{\text{minimize}} \ \ & t \\
%	\text{subject to} \ \ \ \  & t(f_k-f_*) - (f_{k+1}-f_*) = \bfz_0^\top Q_0 \bfz_0 + \sum_{i=1}^m\sigma_i\inn{c_i}{\bfz_0} + \sum_{j=1}^{m'} \theta_j \inn{d_j}{\bfz_0}, \\
%		& 0 = \sum_{\ell=1}^n \bfz_\ell^\top Q_1 \bfz_\ell + \sum_{i=1}^m \sigma_i \inn{C_i}{G} + \sum_{j=1}^{m'} \theta_j \inn{D_j}{G},\\
%		& t\in (0,1),\\
%		& Q \succeq 0, \ \sigma_i \geq 0, \ \theta \in \R,
%\end{aligned}
%\end{equation}
%where $\{\bfz_\ell\}_{\ell=0}^n$ are defined in \eqref{eq:zvar2}.
%%$\bfz_0 = (f_k,f_{k+1},f_*)$, $\bfz_1 = (x_k,x_{k+1},x_*,g_k,g_{k+1},g_*) \in \R^6$, and $Q_0$ and $Q_1$ are blocks of the block-diagonal matrix $Q$.
%The term $\bfz_0^\top Q_0 \bfz_0$ necessarily equals 0, and $\sum_\ell \bfz_\ell^\top Q_1 \bfz_\ell$ may be expressed as $\inn{Q_1}{G}$. \red{Matching coefficients of the monomials $(1,f_k,f_{k+1},\dots)$}, the SDP \eqref{expl:sos_pre_final} may be equivalently written as:
\begin{equation}\label{expl:sos-final}
\begin{aligned}
	\underset{t,\{\sigma_i\}_{i=1}^m, \{\theta_j\}_{j=1}^{m'}}{\text{minimize}} \ \ & t \\
	\text{subject to} \ \  \ \ & \begin{pmatrix}0 \\ 1 \\ -1\end{pmatrix}-t\begin{pmatrix}1 \\ 0 \\-1\end{pmatrix} + \sum_{i=1}^m \sigma_i {c_i} + \sum_{j=1}^{m'} \theta_j {d_j} = 0, \\
		& \sum_{i=1}^m \sigma_i C_i + \sum_{j=1}^{m'} \theta_j D_j \preceq 0,\\
		& t\in (0,1), \ \sigma_i \geq 0, \ \theta_j \in \R.\\
\end{aligned}
\end{equation}
Finally, we note that the constraint $t < 1$ can be dropped if the algorithm is a descent algorithm. The SDP induced by the 1-PEP \eqref{expl:pep-dual} and the SDP induced by the degree-1 SOS problem \eqref{expl:sos-final} are hence equivalent problems. 
\end{proof}

\section{Using the SOS Hierarchy to Obtain New Convergence Bounds}\label{sec:inexact}
In this section, we consider a few variants of GD with inexact line search under both the noisy and noiseless settings.
In ``noisy'' GD, the update step is given by $\bfx_{k+1} = \bfx_k + \gamma_k \bfd_k$, where the error (i.e., the difference between the descent direction  $\bfd_k$ and negative gradient) is bounded relative to the gradient:
\begin{equation}\label{eq:noise}
	\norm{\bfd_k - (-\bfg_k)} \leq \delta \norm{\bfg_k},
\end{equation}
for some noise level $\delta \in [0,1)$. This assumption ensures that the next step taken remains in a descent direction, i.e., $-\bfg_k^\top \bfd_k > 0$, e.g. see  \cite[Page 38]{nonlinear_dimitri}. 

%Since we recover the noiseless version when $\delta = 0$, it suffices to consider the noisy setting.

We begin  by deriving  some inequalities that will be used throughout this section. 
%Recall that $\bfd_k^\top \bfg_k \leq 0$ and $-\bfd_k^\top \bfg_k \geq 0$.
We note that
\begin{align}
\bfd_k^\top \bfg_k &= (\bfd_k + \bfg_k)^\top \bfg_k - \norm{\bfg_k}^2 \notag \\
&\leq \norm{\bfd_k+\bfg_k}\norm{\bfg_k} - \norm{\bfg_k}^2 && \text{by Cauchy-Scwartz} \notag \\
&\leq (\delta-1)\norm{\bfg_k}^2 && \label{eq:conseq1}
\end{align}
where the last inequality follows by \eqref{eq:noise}. By a similar argument, we have
\begin{equation}\label{eq:conseq2}
-\bfd_k^\top \bfg_k = (-\bfd_k - \bfg_k)^\top \bfg_k + \norm{\bfg_k}^2 \leq (\delta + 1)\norm{\bfg_k}^2.
\end{equation}
Squaring and expanding \eqref{eq:noise}, we have
\begin{equation}\label{eq:noise_conseq}
\norm{\bfd_k}^2 \leq -2\bfd_k^\top \bfg_k - (1 - \delta^2)\norm{\bfg_k}^2,
\end{equation}
which, combined  with \eqref{eq:conseq2}, implies that
\begin{equation}\label{eq:conseq3}
\norm{\bfd_k}^2 \leq \left[2(\delta+1) - (1-\delta^2)\right] \norm{\bfg_k}^2 = (\delta+1)^2 \norm{\bfg_k}^2.
\end{equation}
Furthermore, by the triangle inequality, we have $\norm{-\bfg_k} 		\leq \norm{\bfd_k} + \norm{-\bfd_k - \bfg_k}$ and thus
\be\label{eq:conseq4}
 \norm{\bfd_k} \geq \norm{\bfg_k} - \norm{-\bfd_k - \bfg_k} \geq \norm{\bfg_k} - \delta \norm{\bfg_k} = (1-\delta) \norm{\bfg_k}.% 
\ee

%Next, we outline a few of these inexact line search conditions or rules.

%The Wolfe conditions are important for quasi-Newton methods, while the Armijo rule is suited for Newton methods but not for quasi-Newton and conjugate gradient methods. Likewise, the Goldstein rule is frequently used in Newton-type methods but are not suitable for quasi-Newton methods that maintain a positive definite Hessian approximation \cite{wright_numericalopt}.

\subsection{The Armijo Rule}
Using Armijo-terminated line search,  the step size $\gamma_k$ is chosen so that
\begin{align}
f(\bfx_k + \gamma_k \bfd_k) &\leq f(\bfx_k) + \epsilon \gamma_k \bfd_k^\top \bfg_k \label{eq:Ar1} \\
f(\bfx_k + \eta \gamma_k \bfd_k) &\geq f(\bfx_k) + \epsilon \eta \gamma_k \bfd_k^\top \bfg_k,\label{eq:Ar2}
\end{align}
for some $\epsilon \in (0,1)$ and $\eta > 1$, e.g., see  \cite[Section 2.4.1]{nemirovski} and  \cite[Page 29]{nonlinear_dimitri}. 
%In this procedure, we choose the largest step size (up to a factor $\eta$) that fulfils the sufficient decrease condition.
 % In this way, it can dispense with the curvature condition.
In the noisy setting, the gradient  is not available. Substituting $-\bfd_k$ for $\bfg_k$ in \eqref{eq:Ar1}-\eqref{eq:Ar2} we obtain:
\begin{align}
	f(\bfx_k + \gamma_k \bfd_k) &\leq f(\bfx_k) - \epsilon \gamma_k \norm{\bfd_k}^2, \label{eq:GDArN1}\\
	f(\bfx_k + \eta \gamma_k \bfd_k) &\geq f(\bfx_k) - \epsilon \eta \gamma_k \norm{\bfd_k}^2. \label{eq:GDArN2}
\end{align}
When noisy GD with Armijo-terminated line search is applied to an $L$-smooth function, we are able to show the validity of the following inequality:
\begin{equation}\label{eq:armijo_constraint}
	f_k - f_{k+1} - \frac{2\epsilon (1-\delta)^2}{\eta L}\left(\frac{1-\delta}{(1+\delta)^2} - \epsilon \right)\norm{\bfg_k}^2 \geq 0,
\end{equation}
for $\delta \in [0,1)$, $\epsilon \in \left(0,\frac{1-\delta}{(1+\delta)^2}\right)$ and $\eta > 1$. Indeed, 
as $f$ is $L$-smooth we have that
\begin{equation*}
f(\bfy) \leq f(\bfx) + (\bfy-\bfx)^\top \nabla f(\bfx) + \frac{L}{2}\norm{\bfy - \bfx}^2.
\end{equation*}
Substituting  $\bfx = \bfx_k$ and $\bfy = \bfx_k + \eta \gamma_k \bfd_k$ we get
\begin{equation}\label{eq:l-lipschitz}
f(\bfx_k + \eta \gamma_k \bfd_k) - f(\bfx_k) \leq \eta \gamma_k \bfd_k^\top \bfg_k + \frac{L}{2}\eta^2 \gamma_k^2 \norm{\bfd_k}^2.
\end{equation}
%(see Appendix \ref{app:armijo} for details). 
Furthermore, combining \eqref{eq:GDArN2} and \eqref{eq:l-lipschitz}, we have
\begin{align}
0 &\leq \bfd_k^\top \bfg_k + (\epsilon + \frac{L}{2}\eta \gamma_k)\norm{\bfd_k}^2 \notag \\
&\leq (\delta-1)\norm{\bfg_k}^2 + (\epsilon + \frac{L}{2}\eta \gamma_k)(\delta+1)^2\norm{\bfg_k}^2 && \text{by \eqref{eq:conseq1} and \eqref{eq:conseq3}}.\notag 
\end{align}
In turn, this implies 
\be
 \gamma_k \geq \frac{2}{\eta L}\left(\frac{1-\delta}{(1+\delta)^2} - \epsilon \right), \label{eq:GDArN_gamma}
\ee
and as we require $\gamma_k > 0$, we need that $\epsilon < \frac{1-\delta}{(1+\delta)^2}$. Substituting \eqref{eq:GDArN_gamma} into \eqref{eq:GDArN1}, we~have
\begin{align*}
0 &\leq f_k - f_{k+1} - \frac{2\epsilon}{\eta L}\left(\frac{1-\delta}{(1+\delta)^2} - \epsilon \right)\norm{\bfd_k}^2 \\
&\leq f_k - f_{k+1} - \frac{2\epsilon}{\eta L}\left(\frac{1-\delta}{(1+\delta)^2} - \epsilon \right) (1-\delta)^2\norm{\bfg_k}^2, &&
\end{align*}
where the last inequality follows by \eqref{eq:conseq4}. 
We  note  that for $\delta = 0$,  the above inequality reduces to \cite[Equation (3.3.9)]{nemirovski}.

 For the next theorem we use polynomial constraints $h_1(\bfz) \geq 0, \dots, {h_6(\bfz)\geq 0}$ given in~\eqref{eq:interpolability}, and  $h_7(\bfz) \geq 0$ given in \eqref{eq:armijo_constraint}, and search for a degree-1 SOS certificate as described  in \eqref{eq:putinardeg1}. Constructing and solving the appropriate SDP, we obtain the following~result.
\begin{theorem}\label{thm:armijo_result}
	For any $\delta \in [0,1)$, $\epsilon \in \left(0, \frac{1-\delta}{(1+\delta)^2}\right)$ and $\eta > 1$, given an $(\mu,L)$-smooth function $f:\R^n\to \R$ and any sequence of iterates $\{\bfx_k\}_{k\geq 1}$ generated using noisy GD with Armijo-terminated line search, the~bound
	\begin{equation*}
	f_{k+1} - f_* \leq \left[1 - \frac{4\mu\epsilon(1-\delta)^2}{\eta L}\left(\frac{1-\delta}{(1+\delta)^2} - \epsilon \right)\right] (f_k - f_*)
	\end{equation*}
	admits an SOS certificate of degree-1.
\end{theorem}
\begin{proof}
%	Set $h_5(\bfz) \geq 0$ as given in \eqref{eq:interpolability} and $h_7(\bfz) \geq 0$ to be the inequality in \eqref{eq:armijo_constraint}.
	Defining 
	$$\sigma_5 = \frac{4\mu\epsilon (1-\delta)^2}{\eta L}\left( \frac{1-\delta}{(1+\delta)^2} -\epsilon \right), \quad \sigma_7= 1, \quad t = 1 - \frac{4\mu\epsilon(1-\delta)^2}{\eta L}\left(\frac{1-\delta}{(1+\delta)^2} - \epsilon \right),$$  we have that $t(f_k - f_*) - (f_{k+1} - f_*)$ is equal to 
	\begin{equation}\label{eq:main3}
	  \frac{2\epsilon(1-\delta)^2\left( 1-\delta-\epsilon(1+\delta)^2 \right)}{\eta (L-\mu) (1+\delta)^2}\norm{\bfg_k + \mu(\bfx_* - \bfx_k)}^2 + \sigma_5 h_5(\bfz) + \sigma_7 h_7(\bfz).
	\end{equation}
	The first term in the right-hand-side of equation \eqref{eq:main3} is strictly positive since $\epsilon < \frac{1-\delta}{(1+\delta)^2}$. In addition, as previously discussed, any sequence of iterates $\{\bfx_k\}_{k\geq 1}$ generated by noisy GD with Armijo-terminated line search for minimizing a function $f \in \FmuLRn$ satisfies $h_5(\bfz) \geq 0$ and $h_7(\bfz) \geq 0$. Since $\sigma_5, \sigma_7\ge 0$, overall the right-hand-side of \eqref{eq:main3} is positive. Hence, the left-hand-side of equation \eqref{eq:main3} is also positive, concluding the proof.
%	\begin{align*}
%	t(f_k - f_*) - (f_{k+1} - f_*) &= \frac{2\epsilon(1-\delta)^2\left( 1-\delta-\epsilon(1+\delta)^2 \right)}{\eta (L-\mu) (1+\delta)^2}\norm{\bfg_k + \mu(\bfx_* - \bfx_k)}^2 + \sigma_5 h_5(\bfz) + \sigma_7 h_7(\bfz) \\
%	&\geq \frac{2\epsilon(1-\delta)^2\left( 1-\delta-\epsilon(1+\delta)^2 \right)}{\eta (L-\mu) (1+\delta)^2}\norm{\bfg_k + \mu(\bfx_* - \bfx_k)}^2 > 0
%	\end{align*}
%	wbere the strict inequality follows from the fact that $\epsilon < \frac{1-\delta}{(1+\delta)^2}$.
\end{proof}
From Theorem \ref{thm:armijo_result}, we also get  a rate  bound for GD with Armijo-terminated line search in the noiseless case (i.e.,  $\delta = 0$), which is given by
\begin{equation}\label{eq:armijo_result}
f_{k+1} - f_* \leq \left(1 - \frac{4\mu\epsilon(1-\epsilon)}{\eta L}\right) (f_k - f_*).
\end{equation}
As a consequence of \eqref{eq:armijo_result},  for all $N \geq 1$ we have
\begin{equation}\notag
f_N - f_* \leq \left(1 - \frac{4\epsilon(1-\epsilon)}{\eta \kappa}\right)^N (f_0 - f_*).
\end{equation}
where $\kappa = L/\mu$ is the condition number of $f$.
To the best of our knowledge, the best bounds for a function $f \in \FmuLRn$ minimized by GD with Armijo rule were given by Luenberger and Ye \cite[Page~239]{ye} and Nemirovski \cite[Proposition~3.3.5]{nemirovski}. For any $\epsilon < 0.5$ and $\eta > 1$, Luenberger and Ye (LY) showed~that
\begin{equation*}
f_N - f_* \leq \left( 1 - \frac{2\epsilon}{\eta \kappa} \right)^N (f_0 - f_*),
\end{equation*}
while for any $\epsilon \geq 0.5$ and $\eta \geq 1$, Nemirovski showed that
\begin{equation}\notag
f_N - f_* \leq \kappa\left( \frac{\kappa - (2-\epsilon^{-1})(1-\epsilon)\eta^{-1}}{\kappa + (\epsilon^{-1} -1)\eta^{-1}} \right)^N (f_0 - f_*).
\end{equation}
To compare these convergence rates, we consider the three contraction factors
\begin{equation*}
t_{\new} = 1 - \frac{4\epsilon(1-\epsilon)}{\eta \kappa}, \qquad t_{\LY} = 1 - \frac{2\epsilon}{\eta \kappa}, \qquad t_{\nemi} = \frac{\kappa - (2-\epsilon^{-1})(1-\epsilon)\eta^{-1}}{\kappa + (\epsilon^{-1} -1)\eta^{-1}}.
\end{equation*}
Since Luenberger and Ye's bound only holds for $\epsilon \in (0,0.5)$ whereas ours hold for $\epsilon \in (0,1)$, we compare $t_{\new}$ and $t_{\LY}$ within the common range $0 < \epsilon < 0.5$. On the other hand, Nemirovski's bound only holds for $\epsilon \in [0.5, 1)$, hence we compare $t_{\new}$ and $t_{\nemi}$ within this range.

Using  simple arguments  we now show  that $t_{\new} < t_{\LY}$ and $t_{\new} \leq t_{\nemi}$ within each range of comparison.
%By simple algebra as shown in Appendix \ref{app:Armijo_discussion}, it can be shown that $t_{\new} \leq t_{\nemi}$.
Thus, our contraction factor is no larger than those of Luenberger and Ye's, and Nemirovski's. Indeed, to show $ t_{\new} < t_{\LY}$ note that as $\epsilon < 0.5$, we have
\be
1 < 2(1-\epsilon) \iff 
\frac{2\epsilon}{\eta\kappa} < \frac{4\epsilon(1-\epsilon)}{\eta\kappa} \iff 
1 - \frac{ 4 \epsilon (1-\epsilon) }{ \eta \kappa} < 1 - \frac{ 2\epsilon }{ \eta\kappa}.
\ee
Next, we show that $t_{\nemi}  \geq t_{\new}.$ Since $0.5 \leq \epsilon$, we have
\begin{equation*}
1-2\epsilon \geq 1-4\epsilon^2.
\end{equation*}
Since $\eta > 1$, $\kappa > 1$ and $1 - \epsilon > 0$, this implies
\begin{equation*}
\eta\kappa(1-2\epsilon)(1-\epsilon) \geq \eta\kappa(1-4\epsilon^2)(1-\epsilon).
\end{equation*}
Since $4\epsilon(1-\epsilon)^2 > 0$, we can subtract it from the right-hand side:
\begin{equation*}
\eta\kappa(1-2\epsilon)(1-\epsilon) \geq \eta\kappa(1-4\epsilon^2)(1-\epsilon) - 4\epsilon(1-\epsilon)^2,
\end{equation*}
and add $\epsilon(\eta\kappa)^2$ to each side and factorize:
\begin{align*}
\epsilon(\eta\kappa)^2 + \eta\kappa(1-2\epsilon)(1-\epsilon) &\geq \epsilon(\eta\kappa)^2 + \eta\kappa(1-4\epsilon^2)(1-\epsilon) - 4\epsilon(1-\epsilon)^2, \\
\eta\kappa\left[ \epsilon\eta\kappa + (1-2\epsilon)(1-\epsilon) \right] &\geq \left[\epsilon\eta\kappa + (1-\epsilon)\right]\left[ \eta\kappa - 4\epsilon(1-\epsilon) \right].
\end{align*}
Rearranging the equation, we obtain
\begin{align*}
\frac{\epsilon\eta\kappa + (1-2\epsilon)(1-\epsilon)}{\epsilon\eta\kappa + (1-\epsilon)} &\geq \frac{\eta\kappa - 4\epsilon(1-\epsilon)}{\eta\kappa},
\end{align*}
and the proof that  $t_{\nemi}  \geq t_{\new} $  is concluded.

Figures \ref{fig:gdarmijo_compare_luen} and \ref{fig:gdarmijo_compare_nemi} compare $t_{\new}$ with $t_{\LY}$ and $t_{\nemi}$ respectively for various values of $\kappa$ and $\eta$.
\begin{figure}[t]
	\centering
	\begin{minipage}{0.45\textwidth}
		\centering
		\includegraphics[width=\textwidth]{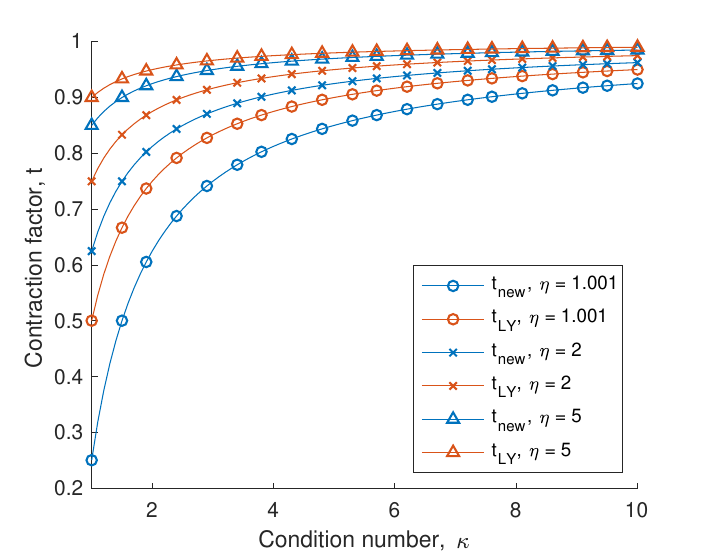}
		\caption{$t_{\new}$ and $t_{\LY}$ for $\epsilon = 0.25$.}
		\label{fig:gdarmijo_compare_luen}
	\end{minipage}
	\begin{minipage}{0.45\textwidth}
		\centering
		\includegraphics[width=\textwidth]{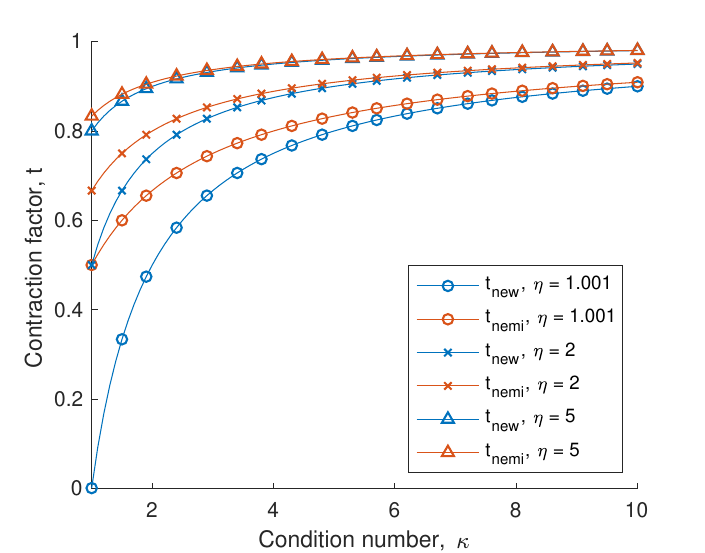}
		\caption{$t_{\new}$ and $t_{\nemi}$ for $\epsilon = 0.5$.}
		\label{fig:gdarmijo_compare_nemi}
	\end{minipage}
\end{figure}
%\begin{figure}[h]
%	\centering
%	\includegraphics[width=0.5\linewidth]{GDArmijo_t_nemi_compare}
%	\caption{Comparison of $t_{\new}$ with $t_{\nemi}$ for $\epsilon = 0.5$.}
%	\label{fig:gdarmijotcomparisonall}
%\end{figure}
We note that when $\eta \to 1^+$ and $\kappa \to 1^+$ and $\epsilon = 0.5$, our contraction factor $t_{\new}$ tends to 0, whereas Nemirovski's contraction factor $t_{\nemi}$ tends to 0.5. In fact, when $\kappa \to 1^+$, the function $f$ behaves roughly as a quadratic. Combining this with the fact that $\eta \to 1^+$ and the updates as shown in \eqref{eq:Ar1}, it can be verified that GD with Armijo rule takes only a single step to attain the optimal solution. Hence, the contraction factor we derived, i.e., $t_{\new} = 0$, is tight in this limiting scenario.

Conversely, Nemirovski's contraction factor $t_{\nemi} = 0.5$ is loose. The looseness of Nemirovski's analysis can be attributed to the fact that he only applies the condition
\begin{equation*}
f(\bfx) + (\bfy - \bfx)^\top\nabla f(\bfx) + \frac{\mu}{2}\norm{\bfy-\bfx}^2 \leq f(\bfy) \leq f(\bfx) + (\bfy - \bfx)^\top\nabla f(\bfx) + \frac{L}{2}\norm{\bfy-\bfx}^2
\end{equation*}
to select iterates (i.e., discretizing the condition), which is not sufficient to guarantee $\FmuL$-interpolability. On the other hand, we make use of the condition from  Theorem \ref{thm:interpolability}, which constitutes a {necessary}  and sufficient condition for a function to be $\FmuL$-interpolable.
%Under these conditions, the function is quadratic and GD with Armijo rule only needs to take ``one step'' to reach optiamlity i.e. $f_1 = f_*$. All subsequent steps remain at the optimal point i.e. $f_{k+1} = f_*$ for $k \geq 0$. Hence, $f_{k+1} - f_* \leq 0\cdot (f_k - f_*)$ holds for all $k$.

\subsection{The Goldstein Rule}
Using Goldstein-terminated line search, the step size is chosen so that
\begin{equation}\label{eq:Gs}
(1-\epsilon)\gamma_k \bfd_k^\top \bfg_k \leq f(\bfx_k + \gamma_k \bfd_k) - f(\bfx_k) \leq \epsilon\gamma_k \bfd_k^\top \bfg_k.
\end{equation}
for some $\epsilon \in (0, 1/2)$ \cite[Section 2.4.2]{nemirovski}. The Goldstein rule was proposed earlier than the Armijo rule, and encapsulates the same principle of sufficient decrease as the Armijo rule \cite[Page 32]{nonlinear_dimitri}.

To examine the performance of the Goldstein-terminated line search in  the noisy setting, we again substitute $-\bfd_k$ for $\bfg_k$ in \eqref{eq:Gs} to get:
\begin{equation}\label{eq:GDGsN}
	-(1-\epsilon)\gamma_k \norm{\bfd_k}^2 \leq f(\bfx_k + \gamma_k \bfd_k) - f(\bfx_k) \leq -\epsilon\gamma_k\norm{\bfd_k}^2.
\end{equation}
When noisy GD with the Goldstein rule is applied to an $L$-smooth function, the following polynomial constraint:
\begin{equation}\label{eq:goldstein_constraint}
	f_k - f_{k+1} - \frac{2\epsilon(1-\delta)^2}{L}\left(\frac{1-\delta}{(1+\delta)^2} - (1-\epsilon)\right)\norm{\bfg_k}^2 \geq 0
\end{equation}
holds for $\delta \in [0,\sqrt{5}-2)$ and $\epsilon \in \left(1-\frac{1-\delta}{(1+\delta)^2},\frac{1}{2} \right)$. 
Indeed, as $f$ is $L$-smooth we~get % along with an $L$-Lipschitz condition
\begin{equation*}
	f(\bfx_k + \gamma_k \bfd_k) - f(\bfx_k) \leq \gamma_k \bfd_k^\top \bfg_k + \frac{L}{2}\gamma_k^2 \norm{\bfd_k}^2,
\end{equation*}
which combined with the first inequality in \eqref{eq:GDGsN} gives 
\begin{align}
	0 &\leq \bfd_k^\top \bfg_k + \left(\frac{L}{2}\gamma_k + (1-\epsilon) \right)\norm{\bfd_k}^2 \notag \\
	  &\leq (\delta-1)\norm{\bfg_k} + \left(\frac{L}{2}\gamma_k + 1-\epsilon \right)(1+\delta)^2 \norm{\bfg_k}^2 && \text{by \eqref{eq:conseq1} and \eqref{eq:conseq3}}\notag.
\end{align}
In turn, this implies that 
	\be \label{eq:GDGsN_gamma}
	 \gamma_k \geq \frac{2}{L}\left( \frac{1-\delta}{(1+\delta)^2} - (1-\epsilon)\right).
\ee
Since we require $\gamma_k > 0$, we need that $\epsilon > 1 - \frac{1-\delta}{(1+\delta)^2}$. Together with the condition that $\epsilon \in (0, 1/2)$, this implies that the Goldstein rule only works in the case where $\delta < \sqrt{5}-2$ and $\epsilon \in \left(1-\frac{1-\delta}{(1+\delta)^2}, \frac{1}{2}\right)$. Substituting \eqref{eq:GDGsN_gamma} into the second inequality in \eqref{eq:GDGsN}, we have
\begin{align*}
	0 &\leq f_k - f_{k+1} - \frac{2\epsilon}{L}\left(\frac{1-\delta}{(1+\delta)^2} - (1-\epsilon)\right)\norm{\bfd_k}^2 \\
	  &\leq f_k - f_{k+1} - \frac{2\epsilon(1-\delta)^2}{L}\left(\frac{1-\delta}{(1+\delta)^2} - (1-\epsilon)\right)\norm{\bfg_k}^2 && \text{by \eqref{eq:conseq4}}
\end{align*}

 %(see Appendix \ref{app:goldstein} for details). 
 
 The polynomial constraints we use in the search of a degree-1 SOS certificate \eqref{eq:putinardeg1} are as follows: $h_1(\bfz) \geq0, \dots, h_6(\bfz)\geq 0$ given in \eqref{eq:interpolability}, as well as \eqref{eq:goldstein_constraint}, denoted by $h_7(\bfz) \geq 0$. As before, after constructing and solving the appropriate SDP we derive the following~result:
\begin{theorem}\label{thm:goldstein_result}
	For any $\delta \in [0,\sqrt{5}-2)$ and $\epsilon \in \left(1-\frac{1-\delta}{(1+\delta)^2},\frac{1}{2} \right)$, given an $(\mu,L)$-smooth function $f:\R^n\to \R$ and any sequence of iterates $\{\bfx_k\}_{k\geq 1}$ generated using noisy GD with the Goldstein rule, the~bound
	\begin{equation*}
	f_{k+1} - f_* \leq \left(1 - \frac{4\mu\epsilon(1-\delta)^2}{L}\left[\frac{1-\delta}{(1+\delta)^2} - (1-\epsilon)\right]\right) (f_k - f_*) 
	\end{equation*}
	admits an SOS certificate of degree-1.
\end{theorem}
\begin{proof}
%	Set $h_5(\bfz) \geq 0$ as given in \eqref{eq:interpolability} and $h_7(\bfz) \geq 0$ to be the inequality in \eqref{eq:goldstein_constraint}.
	Defining $\sigma_5 = \frac{2\mu\epsilon(1-\delta)^2}{L}\left(\frac{1-\delta}{(1+\delta)^2}-(1-\epsilon) \right)$, $\sigma_7= 1$ and 
	$$t = 1 - \frac{2\mu\epsilon(1-\delta)^2}{L}\left(\frac{1-\delta}{(1+\delta)^2}-(1-\epsilon) \right),$$ we have that 
	$t(f_k - f_*) - (f_{k+1} - f_*)$ is equal to
	\begin{equation}\label{eq:main6}
	\frac{2\epsilon(1-\delta)^2\left(\epsilon(1+\delta)^2 - 3\delta-\delta^2\right)}{(L-\mu)(1+\delta)^2}\norm{\bfg_k + \mu(\bfx_* - \bfx_k)}^2 + \sigma_5 h_5(\bfz) + \sigma_7 h_7(\bfz).
		\end{equation}
	The first term in the right-hand-side of equation \eqref{eq:main6} is strictly positive since $\epsilon > 1-\frac{1-\delta}{(1+\delta)^2}$ implies $\epsilon(1+\delta)^2 > \delta^2 + 3\delta$. In addition, any sequence of iterates $\{\bfx_k\}_{k\geq 1}$ generated by noisy GD with the Goldstein rule for minimizing a function $f \in \FmuLRn$ satisfies $h_5(\bfz) \geq 0$ and $h_7(\bfz) \geq 0$. Since $\sigma_5,  \sigma_7 \ge 0$, the right-hand-side of \eqref{eq:main6} is positive. Hence, the left-hand-side of equation \eqref{eq:main6} is also positive, concluding the~proof.
%	\begin{align*}
%	t(f_k - f_*) - (f_{k+1} - f_*) &= \frac{2\epsilon(1-\delta)^2\left(\epsilon(1+\delta)^2 - 3\delta-\delta^2\right)}{(L-\mu)(1+\delta)^2}\norm{\bfg_k + \mu(\bfx_* - \bfx_k)}^2 + \sigma_5 h_5(\bfz) + \sigma_7 h_7(\bfz) \\
%	&\geq \frac{2\epsilon(1-\delta)^2\left(\epsilon(1+\delta)^2 - 3\delta-\delta^2\right)}{(L-\mu)(1+\delta)^2}\norm{\bfg_k + \mu(\bfx_* - \bfx_k)}^2 > 0,
%	\end{align*}
%	where the strict inequality follows from the fact that $\epsilon > 1-\frac{1-\delta}{(1+\delta)^2} \implies \epsilon(1+\delta)^2 > \delta^2 + 3\delta$.
\end{proof}
%\begin{remark}
%	The range of $\delta$ for which we are able to derive a convergence rate is reduced, i.e., we are better able to guarantee the robustness of the Armijo rule to noise than the Goldstein rule.
%\end{remark}
Lastly, from Theorem \ref{thm:goldstein_result}, we recover the bound for GD with the Goldstein rule in the noiseless case (i.e.,  $\delta = 0$), which is given by
\begin{equation*}
	f_{k+1} - f_* \leq \left(1 - \frac{4\mu\epsilon^2}{L}\right) (f_k - f_*).
\end{equation*}
%\begin{corollary}
%	For any $\epsilon \in \left(0,\frac{1}{2} \right)$, given an $(\mu,L)$-smooth function $f:\R^n\to \R$ and any sequence of iterates $\{\bfx_k\}_{k\geq 1}$ generated using GD with the Goldstein rule, the~bound
%	\begin{equation*}
%	f_{k+1} - f_* \leq \left(1 - \frac{4\mu\epsilon^2}{L}\right) (f_k - f_*) 
%	\end{equation*}
%	admits an SOS certificate of degree-1.
%\end{corollary}

\subsection{The Wolfe Conditions}
A step size chosen using the Wolfe conditions must satisfy the following two inequalities:
\begin{align}
f(\bfx_k + \gamma_k \bfd_k) &\leq f(\bfx_k) + c_1 \gamma_k \bfd_k^\top \bfg_k \label{eq:suff_decr}\\
\bfg_{k+1}^\top \bfd_k &\geq c_2 \bfd_k^\top \bfg_k, \label{eq:curvature}
\end{align}
where $\bfd_k$ is a descent direction and $0 < c_1 < c_2 < 1$, e.g. see  \cite[Page 37]{wright_numericalopt}. 

%The first inequality \eqref{eq:suff_decr} ensures that there is sufficient decrease in the objective function value, while \eqref{eq:curvature}, known as the curvature condition, ensures step sizes are not too short.

For the Wolfe conditions, we only have results for noiseless GD. In this setting, $\bfd_k = -\bfg_k$ and hence equations \eqref{eq:suff_decr}-\eqref{eq:curvature} become
\begin{align}
f(\bfx_k - \gamma_k \bfg_k) &\leq f(\bfx_k) - c_1 \gamma_k \norm{\bfg_k}^2, \label{eq:GDWo1}\\
\bfg_{k+1}^\top \bfg_k &\leq c_2 \norm{\bfg_k}^2. \label{eq:GDWo2}
\end{align}
Next, we show that the following polynomial inequality holds when GD with the Wolfe conditions is applied to an $L$-smooth function:
\begin{equation}\label{eq:wolfe_constraint}
	f_k - f_{k+1} - \frac{c_1(1-c_2)}{L} \norm{\bfg_k}^2 \geq 0.
\end{equation}
Indeed, as $f$ is $L$-smooth we have that
\begin{equation*}
\left(\nabla f(\bfx) - \nabla f(\bfy)\right)^\top \left(\bfx - \bfy \right) \leq L \norm{\bfx-\bfy}^2,
\end{equation*}
which for  $\bfx = \bfx_{k+1} = \bfx_k - \gamma_k \bfg_k$ and $\bfy = \bfx_k$ gives
\begin{equation*}
\left(\bfg_{k+1} - \bfg_k \right)^\top \left( -\gamma_k \bfg_k \right) \leq L \gamma_k^2 \norm{\bfg_k}^2.
\end{equation*}
In turn, this implies that 
\begin{align*}
0 &\leq \bfg_k^\top \bfg_{k+1} + \left(L\gamma_k - 1 \right)\norm{\bfg_k}^2 \\
&\leq \left(L\gamma_k - 1 + c_2\right)\norm{\bfg_k}^2, && \text{by \eqref{eq:GDWo2}}
\end{align*}
which shows that 
\be\label{scsvv}
 \gamma_k \geq \frac{1-c_2}{L}.
\ee
Lastly, \eqref{scsvv} together with \eqref{eq:GDWo1}, gives us
\begin{equation*}
0 \leq f_k - f_{k+1} - \frac{c_1(1-c_2)}{L} \norm{\bfg_k}^2.
\end{equation*}
%(see Appendix \ref{app:wolfe} for details). 

In the next theorem, the polynomials we use in the search of a {degree-1} SOS certificate \eqref{eq:putinardeg1} are $h_1(\bfz) \geq 0, \dots, h_6(\bfz)\geq 0$ given in \eqref{eq:interpolability}, as well as~\eqref{eq:wolfe_constraint}, denoted by $h_7(\bfz) \geq 0$. Constructing and solving the appropriate SDP, we obtain the following result.
\begin{theorem}\label{thm:wolfe}
	For any $0 < c_1 < c_2 < 1$, given an $(\mu,L)$-smooth function $f:\R^n\to \R$ and any sequence of iterates $\{\bfx_k\}_{k\geq 1}$ generated using GD with an inexact line search satisfying the Wolfe conditions, the~bound
	\begin{equation*}
	f_{k+1} - f_* \leq \left(1 - \frac{2\mu c_1(1-c_2)}{L}\right) (f_k - f_*)
	\end{equation*}
	admits an SOS certificate of degree-1.
\end{theorem}
\begin{proof}
%	Set $h_5(\bfz) \geq 0$ as given in \eqref{eq:interpolability} and $h_7(\bfz) \geq 0$ to be the inequality in \eqref{eq:wolfe_constraint}.
	Defining $\sigma_5 = \frac{2\mu c_1(1-c_2)}{L}$, $\sigma_7= 1$ and $t = 1 - \frac{2\mu c_1(1-c_2)}{L}$, we have
	\begin{equation}\label{eq:main7}
		t(f_k - f_*) - (f_{k+1} - f_*) = \frac{c_1 (1-c_2)}{(L-\mu)}\norm{\bfg_k + \mu(\bfx_* - \bfx_k)}^2 + \sigma_5 h_5(\bfz) + \sigma_7 h_7(\bfz).
	\end{equation}
	Since $\sigma_5$ and $\sigma_7$ are both nonnegative, and since any sequence of iterates $\{\bfx_k\}_{k\geq 1}$ generated by GD with the Wolfe conditions for minimizing a function $f \in \FmuLRn$ satisfies $h_5(\bfz) \geq 0$ and $h_7(\bfz) \geq 0$, the right-hand-side of \eqref{eq:main7} is nonnegative. The left-hand-side of equation \eqref{eq:main7} is also nonnegative, which concludes the proof.
%	\begin{align*}
%	t(f_k - f_*) - (f_{k+1} - f_*) &= \frac{c_1 (1-c_2)}{(L-\mu)}\norm{\bfg_k + \mu(\bfx_* - \bfx_k)}^2 + \sigma_5 h_5(\bfz) + \sigma_7 h_7(\bfz) \\
%	&\geq \frac{c_1 (1-c_2)}{(L-\mu)}\norm{\bfg_k + \mu(\bfx_* - \bfx_k)}^2 \geq 0.
%	\end{align*}
\end{proof}
\section{Conclusions and Future Work}\label{sec:conc}
This paper proposes a new technique for bounding the convergence rates for various algorithms---namely, by searching for SOS certificates. This leads to a hierarchy of SDPs, for which the first level of the hierarchy is dual to the SDP induced by the one-step PEP as discussed in Section~\ref{sec:duality}.
%It also has several advantages over the PEP framework; it is able to handle more general constraints, and its method of verifying the bound (checking the positive semidefiniteness of a matrix symbolically) is more straightforward than the method suggested by the PEP framework \cite{deKlerk} \cite{ProxGrad}.
Furthermore, using the first level of the SOS hierarchy,  we  %we recover previously-known bounds and 
derive  new bounds for gradient descent  with three popular  inexact line search methods. 

However, our technique does not necessarily produce tight bounds, since it entails two relaxation steps. For one, the constraints characterizing the function class or algorithm may be relaxed. Secondly, we relax the constraint that $p(\bfz)$ be nonnegative to the constraint that $p(\bfz)$ is an SOS. Recall that while SOS implies nonnegativity, the converse is not necessarily true. Proving the tightness of the derived bounds will have to be done via other means.

At present, we have only utilized the first level of the proposed hierarchy by searching for degree-1 certificates.  In  future work, we look to apply the SOS framework to broader function classes, for which exact $\F$-interpolability conditions have not been formulated. In this setting, the SDP formulation of the f-PEP (dual to our degree-1 SOS-SDP) would, in general, not be tight, and going higher up the hierarchy may produce tighter contraction factors (as the degree of the SOS-SDP increases).
%Considering larger function classes beyond the class of $L$-smooth and $\mu$-strongly convex functions would also be useful, since such a property would not hold in most practical applications.
%\red{\sout{In terms of future work, the techniques developed in this work may be applied to other algorithms and function classes for which convergence rates may not be known. In addition, we have focused on the class of functions that are $L$-smooth and $\mu$-strongly convex, a property that does not hold in most practical applications. It would be useful to search for SOS certificates for more general function classes.}} This technique could also be extended to second-order methods, i.e., those that involve second-order derivative information (e.g., Newton's method).
Finally, in the instances where SOS certificates cannot be found, it would be desirable to examine why the technique fails to better understand the scope for which this approach may be~applied.

\subsection*{Acknowledgements}
The authors are supported by a Singapore National Research Foundation (NRF) Fellowship (R-263-000-D02-281).
%\end{acknowledgement}

%This work was supported in part by an NRF Fellowship under grant number R-263-000-D02-281 and a DSTA scholarship. We would like to thank Adrien B.\ Taylor for his valuable comments on our manuscript and for pointing out the duality relation between the first level of the SOS hierarchy and the SDP reformulation of the PEP. We would also like to thank Johan L\"{o}fberg for his useful comments concerning the implementation of the SOS hierarchy using YALMIP.

%\downloadsgraphystyle{IEEEtran}
%\bibliographystyle{spmpsci_unsrt}
\bibliographystyle{abbrv}

\bibliography{spars_bib}
\raggedbottom

\appendix
\end{document}